\documentclass[a4paper,10pt]{article}
\usepackage[utf8]{inputenc}

\usepackage{comment,xcomment,amssymb,amsmath,color}
\usepackage{rotating,xypic,array}
\usepackage{latexsym,mathtools,amsthm}
\usepackage{subfig}
\usepackage{tikz,fancyvrb}
\usepackage{booktabs}  
\usetikzlibrary{arrows.meta}
\usepackage{enumitem}
\usepackage{url} 
\usepackage{longtable}
\usepackage[section]{placeins}
\usepackage{multirow,rotating}

\title{Ricci-flat and Einstein pseudoriemannian nilmanifolds}
\author{Diego Conti and Federico A. Rossi}

\makeatletter
\def\namedlabel#1#2{\begingroup
   \def\@currentlabel{#2}%
   \label{#1}\endgroup
}
\makeatother

\newtheorem{theorem}{Theorem}[section]

\newtheorem{corollary}[theorem]{Corollary}
\newtheorem{proposition}[theorem]{Proposition}
\theoremstyle{definition}
\newtheorem{definition}[theorem]{Definition}
\newtheorem{example}[theorem]{Example}
\newtheorem{question}[theorem]{Question}
\theoremstyle{remark}
\newtheorem{remark}[theorem]{Remark}

\newcommand{\abs}[1]{\left\vert#1\right\vert}
\newcommand{\R}{\mathbb{R}}
\newcommand{\im}{\mathrm{Im}\,}         
\newcommand{\lie}[1]{\mathfrak{#1}}     
\newcommand{\g}{\lie{g}}
\newcommand{\Z}{\mathbb{Z}}

\newcommand{\N}{\mathbb{N}}

\newcommand{\hook}{\lrcorner\,}

\newcommand{\so}{\mathfrak{so}}

\newcommand{\GL}{\mathrm{GL}}
\newcommand{\SL}{\mathrm{SL}}

\newcommand{\id}{\mathrm{Id}}   

\newcommand{\Sl}{\lie{sl}}
\newcommand{\Span}[1]{\operatorname{Span}\left\{#1\right\}}
\newcommand{\tran}[1]{\hspace{.2mm}\prescript{t\hspace{-.5mm}}{}{#1}}

\DeclareMathOperator{\Ric}{Ric}

\DeclareMathOperator{\End}{End}
\DeclareMathOperator{\Aut}{Aut}

\DeclareMathOperator{\ad}{ad}

\DeclareMathOperator{\logsign}{logsign}

\newcolumntype{C}{>{$}c<{$}}
\newcolumntype{L}{>{$}l<{$}}
\newcolumntype{R}{>{$}r<{$}}

\newcommand{\B}{\mathcal B}

\begin{document}
\VerbatimFootnotes
\maketitle

\begin{abstract}
This is partly an expository paper, where the authors' work on pseudoriemannian Einstein metrics on nilpotent Lie groups is reviewed. A new criterion is given for the existence of a diagonal Einstein metric on a nice nilpotent Lie group. Classifications of special classes of Ricci-flat metrics on nilpotent Lie groups of dimension $\leq8$ are obtained. Some related open questions are presented.
\end{abstract}
\renewcommand{\thefootnote}{\fnsymbol{footnote}} 
\footnotetext{\emph{MSC 2010}: 53C25; 53C50, 53C30, 22E25.}
\footnotetext{\emph{Keywords}: Einstein pseudoriemannian metrics, nilpotent Lie groups, nice Lie algebras.}
\footnotetext{This work was partially supported by GNSAGA of INdAM.}
\renewcommand{\thefootnote}{\arabic{footnote}} 

This paper contains an account of our work on pseudoriemannian Einstein metrics on nilpotent Lie groups and some new results, mostly regarding the Ricci-flat case. We restrict to left-invariant metrics, corresponding to scalar products $g$ on the corresponding Lie algebra, also called metrics; the Einstein condition
\begin{equation}
\label{eqn:einstein}
\Ric=\lambda g
\end{equation}is then an algebraic equation in the entries of $g$, though generally quite complicated. A solution to \eqref{eqn:einstein} is called a Ricci-flat metric when $\lambda=0$, and when $\lambda\neq0$ an Einstein metric of nonzero scalar curvature. The nilpotent Lie groups we consider often have rational structure constants, and therefore admit a lattice, i.e. a compact quotient (see \cite{Malcev}); thus, the solutions that we obtain typically determine compact Einstein manifolds.

Examples of Ricci-flat nilpotent Lie algebras appear in the literature in particular contexts: four-dimensional (\cite{Petrov:EinsteinSpaces}), bi-invariant (\cite{DelBarcoOvando,globke,Kath:NilpotentMetric}), nearly parak\"ahler (\cite{ContiRossi:ricci}), $G_2^*$-holonomy (\cite{FinoLujan:TorsionFreeG22}), or $2$-step (\cite{GuediriBinAsfour}). The first example of an  Einstein metric with nonzero scalar curvature on a nilpotent Lie algebra was constructed by the authors in \cite{ContiRossi:EinsteinNilpotent}.

The problem of constructing Einstein nilpotent Lie algebras has no Riemannian counterpart: by~\cite{Milnor:curvatures}, every Riemannian metric on a nonabelian nilpotent Lie algebra has a direction of positive Ricci curvature and a direction of negative Ricci curvature, and cannot therefore be Einstein. Nevertheless, there is a well-established theory of Einstein Riemannian solvmanifolds (see \cite{Lauret:niceEinstein} for a survey), within which some of the techniques we use originated. Indeed, the construction of Riemannian Einstein solvmanifolds is reduced to the study of the Ricci operator on a nilpotent Lie algebra, as they are characterized by the so-called nilsoliton equation (\cite{Lauret:Einstein_solvmanifolds}), involving the Ricci operator of the metric restricted to the nilradical.

Since \cite{Heber:noncompact}, an effective approach used to study the Ricci operator on a nilpotent Lie algebra is to parametrize metric Lie algebras by fixing an orthonormal basis and letting the structure constants vary; the Ricci operator can then be interpreted as a moment map in the sense of geometric invariant theory (\cite{Lauret}). In fact, the Ricci operator can also be viewed as a moment map in the sense of symplectic geometry (see \cite{ContiRossi:EinsteinNilpotent}). This is true for every signature, although the convexity properties exploited in \cite{Lauret} appear not to hold in the indefinite case.


A considerable amount of research has been devoted to the classification of low-dimensional nilsolitons (see \cite{LauretWill:Einstein,KadiogluPayne:Computational,Will:RankOne,FernandezCulma,Nikolayevsky:FreeNilradical,Nikolayevsky:EinsteinDerivation,Arroyo}; more references can be found in   \cite{Lauret:niceEinstein}); most of these results employ, directly or indirectly, the notion of a nice basis. A basis $\{e_1,\dotsc, e_n\}$ of a Lie algebra is called nice if each $[e_i,e_j]$ is a multiple of some $e_h$ and each contraction $e_i\hook de^j$ is a multiple of some element of the dual basis $e^1,\dotsc, e^n$; nice bases were introduced in \cite{LauretWill:Einstein} in the study of nilsolitons, with the observation that $e_1,\dotsc, e_n$ are eigenvectors of the Ricci operator for any metric for which they form an orthogonal nice basis. 

In \cite{ContiRossi:Construction}, we defined a nice Lie algebra as a pair $(\g,\B)$, with $\B$ a nice basis on the Lie algebra $\g$, and an equivalence of nice Lie algebras as an isomorphism that maps basis elements to multiples of basis elements. This is the natural definition for classification purposes,  since the nice condition is unaffected by rescaling any element of the basis.  With this terminology, we have been able to classify nice nilpotent Lie algebras up to dimension $9$. The striking fact is that, at least up to dimension $7$, most nilpotent Lie algebras admit exactly one nice basis up to equivalence (see Theorems~\ref{thm:classification_nice_6} and \ref{thm:classification_nice_7}). This fact was proved in \cite{ContiRossi:Construction} using the classification of nilpotent Lie algebras of dimension $\leq7$ (see \cite{Gong}).

On a nice nilpotent Lie algebra, there are two natural classes of metrics that can be considered: diagonal metrics, that correspond to diagonal matrices in a nice basis, and $\sigma$-diagonal metrics, that correspond to diagonal matrices multiplied by a suitable order two permutation matrix $\sigma$. In this paper we will only consider the case where $\sigma$ is a \emph{diagram involution}, meaning that $[\sigma(e_i),\sigma(e_j)]$ is a multiple of $\sigma([e_i,e_j])$; this condition will be implicitly assumed for the rest of this introduction. It was shown in \cite{ContiRossi:EinsteinNice} that, for any diagram involution $\sigma$,  $\sigma$-diagonal metrics have a diagonal Ricci operator, like diagonal metrics. For both classes, the Einstein equation \eqref{eqn:einstein} reduces to a system of $n$ polynomial equations in $n$ unknowns. 
As $n$ increases, finding a solution (e.g. with a computer algebra system) or proving directly its nonexistence becomes harder. In fact, for $\lambda\neq0$, there is an algebraic obstruction to the existence of an Einstein metric (\cite{ContiRossi:EinsteinNilpotent}); this is only a necessary condition, though it holds for general invariant metrics on nilpotent Lie groups, nice or not. In \cite[Corollary 2.6]{ContiRossi:EinsteinNice} we obtained sharper necessary conditions for the existence of an Einstein metric in the nice diagonal and $\sigma$-diagonal settings. With some computational work, this led to a classification of nice nilpotent Lie algebras of dimension $\leq8$ carrying an Einstein metric of nonzero scalar curvature (see Section~\ref{section:nonzero}).

In this paper we determine conditions on a nice  Lie algebra that are both necessary and sufficient for the existence of an Einstein metric of diagonal or $\sigma$-diagonal type (Theorems~\ref{thm:diagonal} and~\ref{thm:sigmacompatible}). These conditions are still polynomial, but they involve a lower number of parameters and equations than \eqref{eqn:einstein}. We apply this criterion to the case $\lambda=0$, obtaining a classification of diagonal and $\sigma$-diagonal Ricci-flat metrics on nice nilpotent Lie algebras of dimension $\leq8$. In particular, we obtain a one-parameter family of non-isometric Ricci-flat metrics (Example~\ref{example:familyRicciFlatMetrics}).

This paper is organized as follows. The first section reviews the classification of nice nilpotent Lie algebras of dimension $\leq9$ and some open problems in this context. The second section contains a characterization of nice nilpotent Lie algebras admitting an Einstein metric of diagonal or $\sigma$-diagonal type.  The third section reviews our results for the case $\lambda\neq0$ and some related open problems. The final section is dedicated to the Ricci-flat case; it contains a classification of diagonal and $\sigma$-diagonal Ricci-flat metrics on nice nilpotent Lie algebras of dimension $\leq8$, as well as some remarks on the $2$-step case and examples related to parahermitian geometry.

\smallskip
\textbf{Acknowledgements}
We thank Viviana del Barco and the referee for useful suggestions.

\section{Nice Lie algebras}
In this section we survey our work on the classification of nice Lie algebras and state some open questions; for details, we refer to \cite{ContiRossi:Construction}.

The classification of nice Lie algebras is based on linear algebra and combinatorics. We define a \emph{labeled diagram} as a directed acyclic graph (with no multiple arrows) endowed with a function from the set of arrows to the set of nodes; the node so associated to an arrow is called its label. Two labeled diagrams will be regarded as isomorphic if there are compatible bijections between the corresponding nodes, arrows and labels; such a map is called an isomorphism. The group of self-isomorphisms of a labeled diagram $\Delta$ is called its group of automorphisms $\Aut(\Delta)$.

Given a labeled diagram, we write $i\xrightarrow{j}k$ to indicate an arrow from node $i$ to node $k$ labeled by the node $j$. We write $i\xrightarrow{j,k} v$ if there exists an $l$  such that $j\xrightarrow{k} l$ and $i\xrightarrow{l}v$ are arrows.

A labeled diagram is called a \emph{nice diagram} if the following hold:
\begin{enumerate}[label=(N\arabic*)]
\item\label{enum:condNice1} any two distinct arrows with the same source have different labels;
\item\label{enum:condNice2} any two distinct arrows with the same destination have different labels;
\item\label{enum:condNice3} if $i\xrightarrow{j}k$ is an arrow, then $i$ differs from $j$ and $j\xrightarrow{i}k$ is also an arrow;
\item\label{enum:condNice4} there do not exist four different nodes $i,j,k,v$ such that exactly one of
\[i\xrightarrow{j,k} v, \quad j\xrightarrow{k,i} v,\quad  k\xrightarrow{i,j} v\]
holds.
\end{enumerate}

Let $\g$ be a lie algebra; let $\B=\{e_1,\dotsc, e_n\}$ be a basis, and denote by $\B^*=\{e^1,\dotsc, e^n\}$ its dual basis. We say that $\B$ is \emph{nice} if 
\begin{itemize}
\item for any $e_i,e_j\in \B$, $[e_i,e_j]$ is a multiple of some element of $\B$;
\item for any $e_i\in\B$, $e^j\in\B^*$, $e_i\hook de^j$ is a multiple of some element of $\B^*$.
\end{itemize}
It is clear that rescaling one or more basis elements does not affect this definition, motivating the following:
\begin{definition}
A \emph{nice Lie algebra} is a pair $(\g,\B)$, where $\g$ is a Lie algebra and $\B$ a nice basis of $\g$. Two nice Lie algebras $(\g_1,\B_1)$ and $(\g_2,\B_2)$ are \emph{equivalent} if there exists a Lie algebra isomorphism $\g_1\cong\g_2$ that maps each element of $\B_1$ to a  multiple of an element of $\B_2$; such a map is called an \emph{equivalence}. 
\end{definition}
As a matter of notation, we will use a string of the form 
\begin{equation}
\label{eqn:reducible}
\texttt{62:4a}\quad (0,0,0,0,e^{13}+e^{24},e^{12}+e^{34})
\end{equation}
to indicate a Lie algebra with a basis $\B=\{e_1,\dotsc, e_6\}$ such that
\[de^1=0=\dots = de^4, \quad de^5=e^{13}+e^{24}, \quad de^6=e^{12}+e^{34}\]
(where, as usual, we have written $e^{ij}$ in lieu of $e^i\wedge e^j$); the label \texttt{62:4a} is the name of this Lie algebra in the classification of \cite{ContiRossi:Construction}.

A nice Lie algebra $(\g,\B)$ is \emph{reducible} if there exist nice Lie algebras $(\g_1,\B_1)$, $(\g_2,\B_2)$ such that $\g=\g_1\oplus \g_2$ and $\B=\B_1\cup \B_2$, and \emph{irreducible} otherwise. Notice that an irreducible nice Lie algebra may be reducible in the category of Lie algebras. For instance, the nice Lie algebra~\eqref{eqn:reducible} is isomorphic, but not equivalent, to the reducible nice Lie algebra 
\[\texttt{62:2} \quad (0,0,e^{12},0,0,e^{45}).\]

In this paper, we will only be interested in the case where $\g$ is nilpotent. To each nice nilpotent Lie algebra $(\g,\B)$ we can associate a nice diagram $\Delta$ by the following rules:
\begin{itemize}
\item the nodes of $\Delta$ are the elements of the nice basis $\B$;
\item there is an arrow $e_i\xrightarrow{e_j} e_k$ if $e_k$ is a nonzero multiple of $[e_i,e_j]$.
\end{itemize}
It is clear that equivalent nice Lie algebras determine isomorphic diagrams.

Conversely, nice diagrams can be used to construct nice nilpotent Lie algebras; however, this requires fixing some additional data. Given a nice diagram $\Delta$ with nodes $1,\dots, n$, let $\{e_i\}$ be the standard basis of $\R^n$ and $\{e^i\}$ its dual basis, and let $V_\Delta\subset \Lambda^2(\R^n)^*\otimes \R^n$ be the vector space  spanned by $e^{ij}\otimes e_k$, where $i\xrightarrow{j}k$ ranges among arrows of $\Delta$. We can parametrize the  $e^{ij}\otimes e_k$ by the index set $\mathcal{I}_\Delta$ that contains $\{\{i,j\},k\}$ whenever $i\xrightarrow{j}k$ is an arrow. 

The generic element of $V_\Delta$ has the form  
\[c=\sum_{I\in \mathcal{I}_\Delta} c_I e^I.\]
We shall also write $c=\sum c_{ijk}e^{ij}\otimes e_k$, where $\{\{i,j\},k\}$, $i<j$ ranges in $\mathcal{I}_\Delta$. Let $\mathring V_\Delta$ be the open subset of $V_\Delta$ where each coordinate $c_{I}$ is nonzero.

\begin{proposition}[\protect{\cite[Proposition 1.3]{ContiRossi:Construction}}]
For any $c$ in $V_\Delta$, suppose that the derivation $d$ of $\Lambda(\R^n)^*$ given by  $de^k=\sum c_{ijk}e^{ij}$ satisfies $d^2=0$, thereby defining a Lie algebra $\g=(\R^n,d)$. Then $(\g,\{e_i\})$ is nice and nilpotent. If in addition $c\in\mathring V_\Delta$, the nice diagram associated to this Lie algebra is $\Delta$.
\end{proposition}
On $V_\Delta$, there is a natural action of $\Aut(\Delta)$; explicitly, if $\sigma$ is a permutation of the nodes that belongs to $\Aut(\Delta)$, we write
\begin{equation}
 \label{eqn:tildesigma}
\tilde\sigma(e^{ij}\otimes e_k)=e^{\sigma_i}\wedge e^{\sigma_j}\otimes e_{\sigma_k}.
\end{equation}
In addition, the Lie group $D_n$ of invertible diagonal matrices of order $n$ acts naturally on $V_\Delta$. By construction, two elements $c,c'$ of $V_\Delta$ define equivalent Lie algebras if they are in the same $\Aut(\Delta)\ltimes D_n$-orbit.

Let $\g$ be nice with diagram $\Delta$ and fix an ordering on the index set $\mathcal{I}_\Delta$. In this paper we will use lexicographic ordering, obtained by associating to each $\{\{i,j\},k\}\in \mathcal{I}_\Delta$ with $i<j$ the triplet $(k,i,j)$, i.e.
 \[\{\{i,j\},k\}< \{\{l,m\},h\} \iff (k < h) \vee( k=h\wedge i<l),\quad i<j, l<m.\]
Notice that by \ref{enum:condNice2}, the two elements of $\mathcal{I}_\Delta$ coincide when $k=h$ and $i=l$.

The action of $D_n$ on $V_\Delta$ has weight vectors $e^{ij}\otimes e_k$; we denote by $M_\Delta$ the matrix whose rows are the weights for this action, following the ordering of $\mathcal{I}_\Delta$. Explicitly, if $E_{11},\dotsc, E_{nn}$ is the canonical basis of $d_n$, we have 
\[(x_1E_{11}+\dots + x_nE_{nn})(e^{ij}\otimes e_k)=(x_k-x_i-x_j)e^{ij}\otimes e_k,\]
so the weight of $e^{ij}\otimes e_k$ is the linear map
\[x_1E_{11}+\dots + x_nE_{nn}\mapsto x_k-x_i-x_j.\]
Up to a sign convention,  $M_\Delta$ is known as the \emph{root matrix} in the literature.
\begin{example}
In the case of \eqref{eqn:reducible}, we have
\[M_\Delta=
 \begin{pmatrix}
-1& 0 &-1&0 & 1 & 0 \\
0&-1& 0 &-1& 1 & 0 \\
-1& -1 &0&0 & 0 & 1 \\
0&0&-1& -1 &0 & 1 
 \end{pmatrix}.
\]
The first row, for instance, corresponds to the weight vector $e^{13}\otimes e_5$; in other words, to the fact that $[e_1,e_3]$ is a nonzero multiple of $e_5$. The other rows are obtained in the same way.
\end{example}

The root matrix is a useful tool in the construction of Einstein metrics, since the existence of a Riemannian nilsoliton metric on a nice Lie algebra only depends on the root matrix (\cite{Nikolayevsky}). It is also useful in the  classification of nice Lie algebras; for instance,  \cite{KadiogluPayne:Computational} used the root matrix to classify nice nilpotent Lie algebras with invertible root matrix and simple pre-Einstein derivation in dimensions $\leq 8$. In this case, any two elements of $\mathring V_\Delta$ define equivalent Lie algebras and the group $\Aut(\Delta)$ is trivial.

In the general case, one needs to study the set of $\Aut(\Delta)\ltimes D_n$-orbits  in $V_\Delta$; since $\Aut(\Delta)$ is discrete, it is natural to break the problem in two and describe a section for the action of $D_n$ first. One possible way to do so is described in \cite{Payne:Methods}, by taking a subspace $\Delta_0^p$ that corresponds to a linear subspace in logarithmic coordinates. Our approach is to consider a linear subspace in the coordinates $c_{I}$.

By way of notation, since the entries of $M_\Delta$ are integers, we can define its reduction $\operatorname{mod} 2$, that will be indicated by $M_{\Delta,2}$.

\begin{proposition}
[\protect{\cite[Proposition 2.2]{ContiRossi:Construction}}]
\label{prop:fundamental_domain}
Choose $\mathcal{J}_{\Delta,2}\subset\mathcal{J}_{\Delta}\subset \mathcal{I}_\Delta$ so that $\mathcal{J}_{\Delta,2}$ parametrizes a maximal set of $\Z_2$-linearly independent rows of $M_{\Delta,2}$ and $\mathcal{J}_{\Delta}$ parametrizes a maximal set of $\R$-linearly independent rows of $M_{\Delta}$. Set
\[
W=\left\{\sum c_I E_I\in V_\Delta\mid c_I=1\ \forall I\in \mathcal{J}_{\Delta,2},\ c_I= \pm 1 \ \forall I\in \mathcal{J}_{\Delta}\setminus \mathcal{J}_{\Delta,2}\right\}.\]
Then $\mathring{W}=W\cap  \mathring{V}_\Delta$ is a fundamental domain in $\mathring{V}_\Delta$ for the action of $D_n$.
\end{proposition}
Considering now the full group $\Aut(\Delta)\ltimes D_n$, it is not difficult to compute the action of $\Aut(\Delta)$ on the set of connected components of $\mathring W$, and choose connected components $W_1,\dotsc, W_k$, one in each orbit.  By construction, elements of different families are in different  $\Aut(\Delta)\ltimes D_n$-orbits (so they cannot define equivalent nice Lie algebras), although a single $W_i$ may contain two points in the same orbit. Finally imposing the quadratic equations corresponding to the Jacobi equality, one obtains (at most) $k$ inequivalent families of nice Lie algebras with diagram $\Delta$.

Implementing this strategy with a computer (see \url{https://github.com/diego-conti/DEMONbLAST}), we obtained:
\begin{theorem}[\protect{\cite[Proposition 3.1]{ContiRossi:Construction}}]
\label{thm:classification_nice_6}
Among the 34  nilpotent Lie algebras of dimension $6$:
\begin{itemize}
\item one does not admit any nice bases;
\item 3 admit exactly two inequivalent nice bases;
\item the remaining 30 admit exactly one nice basis up to equivalence.
\end{itemize}
\end{theorem}

\begin{figure}[thp]
\caption{\label{fig:inequivalentdiagrams6}
Nice diagrams associated to inequivalent nice bases on the Lie algebra $N_{6,2,5}$
}
\includegraphics[width=0.4\textwidth]{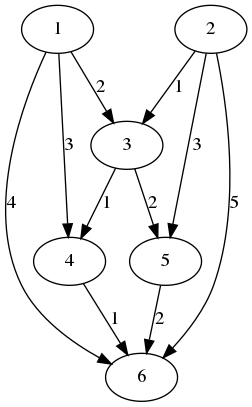}\hfill
\includegraphics[width=0.4\textwidth]{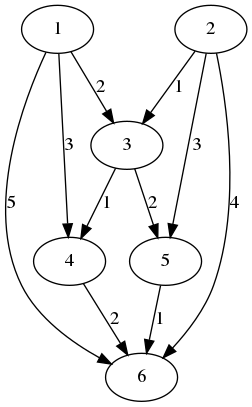}
\end{figure}

The analogous statement in dimension 7 is complicated by the fact that continuous families of Lie algebras appear.
\begin{theorem}[\protect{\cite[Theorem 3.2]{ContiRossi:Construction}}]
\label{thm:classification_nice_7}
Among the 175 isolated nilpotent Lie algebras of dimension $7$:
\begin{itemize}
\item 34 does not admit any nice bases;
\item 11 admit exactly two inequivalent nice bases;
\item the remaining 130 admit exactly one nice basis up to equivalence.
\end{itemize}
Among the $9$ families of nilpotent Lie algebras of dimension $7$, with reference to the notation of \cite{Gong},
\begin{itemize}
 \item
 $12457N$, $12457N_2$, $1357N$, $1357S$ (for $\lambda\leq 0$), $147E_1$ (for $\lambda>1, \lambda\neq2$)
 do not admit a nice basis;
\item $1357QRS_1$ for $\lambda= 1$ admits exactly two inequivalent nice bases;
\item $1357S$ for $\lambda> 0$, $147E_1$ (for $\lambda=2$), $147E$, $123457I$, $1357QRS_1$ (for $\lambda\neq1$), $1357M$ admit exactly one nice basis up to equivalence.
\end{itemize}
\end{theorem}
Extending these results to higher dimensions is made difficult by the fact that nilpotent Lie algebras are not classified, although with the same methods we have been able to classify \emph{nice} nilpotent Lie algebras of dimension $8$ and $9$ (see \cite{ContiRossi:Construction}). Nevertheless, it is natural to ask whether the low-dimensional behaviour generalizes.

\begin{question}
\label{question:f}
For fixed $n\in\N$, are there nilpotent Lie algebras of dimension $n$ with an infinite number of inequivalent nice bases? What is the largest number $f(n)\in\N\cup\{+\infty\}$ of inequivalent nice bases that can be found on a nilpotent Lie algebra of dimension $n$?
\end{question}

By the above-mentioned result, we know that $f(n)=1$ for $n\leq5$ and $f(6)=2=f(7)$. In addition, in \cite[Corollary 3.7]{ContiRossi:Construction}, we proved that the number of inequivalent nice bases on a fixed nilpotent Lie algebra is at most countable.

\begin{proposition}
The function $f\colon \N\to \N\cup\{+\infty\}$ defined in Question~\ref{question:f} is nondecreasing and unbounded; more precisely,
\[f(n)\geq \left[\frac n6\right]+1.\]
\end{proposition}
\begin{proof}
We first show that $f$ is nondecreasing. Given a Lie algebra $\g$ with center $Z$ and derived Lie algebra $\g'$, we can decompose each nice basis as
\[\B=\B_0\cup \B_+,\]
where $\B_0=\B\cap (Z\setminus\g')$ and $\B_+$ is its complement; this corresponds to decomposing the nice diagram into the union of the subgraph of isolated vertices and the subgraph of vertices of positive degree.

It is clear that two bases $\B$ and $\B'$ on $\g$ are equivalent if and only if $\B_+$ is equivalent to $\B'_+$. Therefore, if $\B$ and $\B'$ are inequivalent nice bases on a nilpotent Lie algebra $\g$ of dimension $n$, $\B\cup\{e_{n+1}\}$ and $\B'\cup\{e_{n+1}\}$ are inequivalent nice bases on $\g\oplus
\R$. This shows that $f(n+1)\geq f(n)$.

To see that $f$ is unbounded, let $\g$ be the nilpotent Lie algebra denoted by $N_{6,2,5}$ in \cite{Gong}. As shown in \cite{ContiRossi:Construction}, $\g$ has two inequivalent nice bases; the associated nice diagrams $\Delta_1$ and $\Delta_2$ (see Figure~\ref{fig:inequivalentdiagrams6}) are connected and not isomorphic.

Given $h,k\in\N$, we can define a nice diagram
$\Delta_{hk}$
by adjoining $h$ copies of $\Delta_1$ and $k$ copies of $\Delta_2$; this determines a nice basis on the nilpotent Lie algebra \[\underbrace{\g\oplus \dots\oplus \g}_{h+k}.\] Two such nice bases can only be equivalent if the underlying nice diagrams $\Delta_{hk}$ and $\Delta_{h'k'}$ are isomorphic; in turn, this implies $h=h'$ and $k=k'$, because diagram isomorphisms map connected components to connected components.

This shows that $f(6n)\geq n+1$; since $f$ is nondecreasing, the statement follows.
\end{proof}
Another striking consequence of the classification is that two nice bases on a fixed nilpotent Lie algebra of dimension $\leq 7$ with isomorphic diagrams are always equivalent; thus, a nilpotent Lie algebra of dimension $\leq7$ has as many nice diagrams as inequivalent nice bases. It is then natural to ask:

\begin{question}
How many nonisomorphic nice diagrams can a nilpotent Lie algebra have?
\end{question}

\begin{question}
How many inequivalent nice bases with the same nice diagram can a nilpotent Lie algebra have?
\end{question}


\section{Einstein metrics on nice Lie algebras}
\label{sec:einsteinonnice}
Nice Lie algebras lend themselves to the construction of Einstein metrics. In this section we review the method of  (\cite[Section 2]{ContiRossi:EinsteinNice}) to construct Einstein pseudoriemannian metrics on a nice nilpotent Lie algebra and provide a new condition on a fixed Lie algebra to determine whether the method can be applied.

We are interested in left-invariant metrics on a Lie group $G$, which can be identified with scalar products on its Lie algebra $\g$; such a scalar product will be called a \emph{metric} on $\g$. We shall consider two distinct classes of metrics, namely \emph{diagonal} and \emph{$\sigma$-diagonal} metrics.

By our definition, nice Lie algebras are endowed with a nice basis $\B$; a metric on a nice Lie algebra is \emph{diagonal} if its basis is orthogonal. Fixing an order in the basis, we define the \emph{signature} of a diagonal metric $\sum g_i e^i\otimes e^i$ as the vector $(\logsign g_i)\in\Z_2^n$, where
\[\logsign x=\begin{cases} 0 & x>0 \\ 1 & x<0\end{cases},\]
the notation being justified by the identity $(-1)^{\logsign x}=\operatorname{sign}(x)$. We shall also write $\logsign g$ for the vector with entries $\logsign g_i$.

If a nice Lie algebra $\g$ is reducible (in the nice category), a diagonal metric on $\g$ is the direct sum of diagonal metrics on its factors; geometrically, this situation corresponds to a product metric. In particular, if the diagonal metric on  $\g$ is Einstein, so is the metric on each factor;
for this reason, diagonal metrics are most interesting when the nice Lie algebras are irreducible.

For $v=(v_1,\dotsc, v_n)\in\R^n$, let $v^D\in d_n$ be the diagonal matrix with entry $v_i$ at $(i,i)$. The root matrix determines a homomorphism of abelian Lie algebras
\[M_\Delta^D\colon d_n\to d_m, \qquad M_\Delta^D(v^D)=(M_\Delta(v))^D,\]
which is the differential at the identity of the Lie group homomorphism
\[e^{M_\Delta}\colon D_n\to D_m, \qquad e^{M_\Delta}(g^D)
 =\bigl(\prod_{j=1}^n g_j^{(M_\Delta)_{1j}},\dots,\prod_{j=1}^n g_j^{(M_\Delta)_{mj}}\bigr)^D
\]

Diagonal metrics are naturally parametrized by elements $g=(g_1,\dotsc, g_n)^D\in D_n$. It turns out (see \cite{LauretWill:Einstein}) that the Ricci tensor of a diagonal metric is again diagonal; for an explicit formula, we will use the following:
\begin{proposition}[\cite{ContiRossi:EinsteinNice}]
\label{prop:ricci}
Let $g$ be a diagonal metric on a nice Lie algebra with diagram $\Delta$ and structure constants $c$. Define $X$ by 
\[X^D= e^{M_\Delta}(g)(c^D)^2.\]
Then the Ricci operator is given by
\[\Ric = \frac12 (\tran{M_\Delta} X)^D.\]
\end{proposition}
The Einstein equation with cosmological constant $\frac12k$ then reads $(\tran{M_\Delta} X)^D=k\id$; we shall denote by $[k]$ the vector in $\R^m$ with all entries equal to $k$, and equivalently  write
\[\tran{M_\Delta} X=[k].\]
Given real vectors $X=(x_1,\dotsc, x_m)$ and $\alpha=(\alpha_1,\dotsc, \alpha_m)$, in usual multiindex notation we shall write $\abs{X}^\alpha=\prod_{j=1}^m\abs{x_j}^{\alpha_j}$.

The existence of a diagonal Einstein metric can be determined via the following:
\begin{theorem}
\label{thm:diagonal}
Let $\g$ be a nice Lie algebra with diagram $\Delta$ and structure constants $c\in V_\Delta$. Then $\g$ has a diagonal metric of signature $\delta$ satisfying $\Ric=\frac12k\id$, $k\in\R$ if and only if for some $X\in \R^m$:
\begin{description}
\item[($\mathbf{K}$)\namedlabel{enum:condK}{($\mathbf{K}$)}] $\tran M_\Delta X=[k]$:
\item[($\mathbf{H}$)\namedlabel{enum:condH}{($\mathbf{H}$)}] $X$ does not belong to any coordinate hyperplane;
\item[($\mathbf{L}$)\namedlabel{enum:condL}{($\mathbf{L}$)}] $\logsign X=M_{\Delta,2}\delta$;
\item[($\mathbf{P}$)\namedlabel{enum:condP}{($\mathbf{P}$)}] for a basis $\alpha_1,\dotsc, \alpha_k$ of $\ker \tran{M_\Delta}$, we have
\[\abs{X}^{\alpha_i}=\abs{c}^{2\alpha_i}, \quad i=1,\dotsc, k.\]
\end{description}
\end{theorem}
\begin{proof}
Given $\delta\in(\Z_2)^n$, write $(-1)^\delta$ for the matrix $((-1)^{\delta_1},\dots,(-1)^{\delta_n})^D$. A generic element of $D_n$ has the form $g=(-1)^\delta \exp v$, with $v\in d_n$. 
We then have
\[e^{M_\Delta}(g)=(-1)^{M_{\Delta,2}(\delta)}\exp (M_\Delta(v)^D).\]
The metric $g$  has signature $\delta$; it solves $\Ric=\frac12k\id$ if 
\[\tran{M_\Delta} X=[k], \quad X^D=(-1)^{M_{\Delta,2}(\delta)}\exp (M_\Delta(v)^D)(c^D)^2;\]
equivalently, $X$ is characterized by
\[\logsign X=M_{\Delta,2}\delta, \quad \abs{X}^D=\exp (M_\Delta(v)^D)(c^D)^2.\]
Taking componentwise logarithms, we find
\begin{equation}
\label{eqn:Xcsquaredv}
\log \abs{X}- 2\log\abs{c} = M_\Delta(v),
\end{equation}
where the left-hand side denotes a vector with entries $\log \abs{x_i}-2\log\abs{c_i}$.
Since $\im M_\Delta$ is $\ker \alpha_1\cap \dotsc\cap \ker \alpha_k$, \eqref{eqn:Xcsquaredv} has a solution in $v$ if and only if 
\[\alpha_i(\log \abs{X}-\log c^2)=0, \quad i=1,\dotsc,k,\]
which is equivalent to Condition~\ref{enum:condP}.
\end{proof}
A similar result was presented in \cite[Theorem 2.3]{ContiRossi:EinsteinNice}, which characterized the existence of the metric in terms of the existence of a solution to the polynomial equation $e^{M_\Delta}(g)=X$, with $X$ as in \ref{enum:condK}; in practice, applying this criterion amounts to solving a polynomial system of $n$ equations in $n$ unknowns. The improvement of Theorem~\ref{thm:diagonal} is that \ref{enum:condP} corresponds to a polynomial system  of $k$ equations in $k$ unknowns, and $k$ is typically less than $n$; for instance when $n=8$, case-by-case computations show that the maximum value of $k$ is $5$.

\begin{remark}
\label{rem:ActionKerMDelta}
By construction, given $X$ as in Theorem~\ref{thm:diagonal}, the metric $g$ is obtained by solving $X^D= e^{M_\Delta}(g)(c^D)^2$. In particular, $X$ determines the metric uniquely up to the kernel of $e^{M_\Delta}$.

To understand this ambiguity in the choice of $g$, identify $\g$ with $\R^n$ by fixing an order in the nice basis. Then $g=(g_1,\dotsc, g_n)^D\in D_n$ defines a Lie algebra automorphism if for any nonzero bracket $[e_i,e_j]=c_{ijk}e_k$ one has $[ge_i,ge_j]=c_{ijk}ge_k$, i.e. $\frac{g_k}{g_ig_j}=1$. This holds precisely when $e^{M_\Delta}(g)=\id$; therefore, $\ker e^{M_\Delta}$ coincides with the group of diagonal automorphisms of $\g$.

Thus, when $e^{M_\Delta}(g)=e^{M_\Delta}(h)$, we obtain a Lie algebra automorphism  $f=gh^{-1}\colon \g\to\g$. If $g$ and $h$ lie in the same connected component of $D_n$ (in particular, they have the same signature), then we can write $f=\exp t$ so that $\exp t/2$ defines an isometry between the metric Lie algebras $(\g,\sum g_i(e^i)^2)$ and $(\g,\sum h_i(e^i)^2)$. Thus, $X$ determines the metric in an essentially unique way, at least for fixed signature; we refer to \cite{ContiRossi:Construction} for more details.
\end{remark}

\begin{example}\label{es:754321:9}
We illustrate the procedure in the example of the one-parameter families of nice Lie algebras 
\[\texttt{754321:9} \quad (0,0,(1-\lambda) e^{12},e^{13},\lambda e^{14}+e^{23},e^{15}+e^{24},e^{16}+e^{25}+e^{34}).\]
In this case there are no Einstein metrics of nonzero scalar curvature because there exist derivations with nonzero trace (see~Theorem~\ref{th:OstruzioneEinstein}); in terms of Theorem~\ref{thm:diagonal}, this is reflected in the fact that $M_\Delta X=[1]$ has no solution. For $k=0$, $M_\Delta X=0$ has solutions of the form
\[X=(x_9,x_8+x_9+x_7,x_8,-x_8-x_9-x_7,-x_8-x_9-x_7,x_7,-x_8-x_9,x_8,x_9).\]
The structure constants are $c=(\lambda,1,1-\lambda,1,\dots, 1)$. 
Condition~\ref{enum:condP} gives
\begin{equation}
 \label{eqn:22487111x}
\begin{gathered}
{\left| \frac{x_9^{2}}{ {(x_9+x_7+x_8)} {(x_9+x_8)}}\right|}={(-1+\lambda)}^{2},\quad {\left|\frac{x_7}{x_9+x_7+x_8}\right|}=1,\\
{\left|\frac{x_8^{2}}{ {(x_9+x_7+x_8)} {(x_9+x_8)}}\right|}=\lambda^{2}.
\end{gathered}
\end{equation}
Condition~\ref{enum:condH} implies $x_8+x_9\neq0$, so the second equation in \eqref{eqn:22487111x} is satisfied if $x_7=-x_7-x_8-x_9$, i.e. 
\[X=(x_9,-x_7,x_8,x_7,x_7,x_7,2x_7,x_8,x_9).\]
This is sufficient to prove that this nice Lie algebra has no diagonal Ricci-flat metric, since  $\logsign X$ is not in $\im M_{\Delta,2}$.

For comparison, we note that applying \cite[Theorem 2.3]{ContiRossi:Construction} (or Proposition~\ref{prop:ricci}) shows the Ricci-flat condition to be equivalent to the system
\begin{equation}
 \label{eqn:22487111g}
\begin{gathered}
\frac{g_3}{g_1g_2}=\frac{x_9}{(1-\lambda)^2}, \qquad 
\frac{g_4}{g_1g_3}=x_8+x_9+x_7, \qquad 
\frac{g_5}{g_1g_4}=\frac{x_8}{\lambda^2},\\
\frac{g_5}{g_2g_3}=-x_8-x_9-x_7,\qquad 
\frac{g_6}{g_1g_5}=-x_8-x_9-x_7, \qquad 
\frac{g_6}{g_2g_4}=x_7, \\
\frac{g_7}{g_1g_6}=-x_8-x_9, \qquad 
\frac{g_7}{g_2g_5}=x_8,\qquad
\frac{g_7}{g_3g_4}=x_9.
\end{gathered}
\end{equation}
Equation~\eqref{eqn:22487111x} is obtained by eliminating the $g_i$ from this system, and the condition on $\logsign X$ means that, after imposing \eqref{eqn:22487111x}, it is not possible to choose the signs of the $g_i$ consistently in order to satisfy \eqref{eqn:22487111g}.
\end{example}
\begin{remark}
In the above example, it is possible to solve separately $\abs{X}^{\alpha_i}=\abs{c}^{2\alpha_i}$ and $\logsign X=M_{\Delta,2}(\delta)$, with $X$ in $\ker \tran{M_\Delta}$; the essential fact is that the two equations cannot be solved simultaneously.
\end{remark}

The second class of metrics that we consider is that of $\sigma$-diagonal metrics. Given a permutation of order two $\sigma\in\Sigma_n$, we say that a $\sigma$-diagonal metric is a scalar product of the form
\[\sum g_ie^i\otimes e^{\sigma_i},\]
where $g$ is a $\sigma$-invariant element of $(\R^*)^n$. By construction, $\logsign g$ is always an element of $((\Z_2)^n)^\sigma$, i.e. a $\sigma$-invariant element of $(\Z_2)^n$. Notice that the signature of the scalar product depends on both $\logsign g$ and $\sigma$.

We will consider the case where $\sigma$ is a \emph{diagram involution}, i.e. an element of $\Aut(\Delta$); we therefore have a commutative diagram
\[
 \xymatrix{
 \R^n\ar[r]^{M_\Delta}\ar[d]^\sigma & \R^m\ar[d]^\sigma \\
 \R^n\ar[r]^{M_\Delta} & \R^m}
\]
Both the endomorphisms labeled by $\sigma$ are symmetric of order two. In addition we have a linear map $\tilde\sigma\colon V_\Delta\to V_\Delta$ defined by \eqref{eqn:tildesigma}; denoting by $c\in V_\Delta$ the structure constants vector, we set 
\[\tilde c=\tilde\sigma (c)=\sum c_{ijk}e^{\sigma_i,\sigma_j}\otimes e_{\sigma_k}.\]
A straightforward generalization of Proposition~\ref{prop:ricci} gives
\begin{proposition}[\cite{ContiRossi:EinsteinNice}]
\label{prop:ricciSigma}
Let $\Delta$ be a nice diagram with a diagram involution $\sigma$, and let $g$ be a $\sigma$-diagonal metric on a nice Lie algebra with diagram $\Delta$ and structure constants $c$. Define $X$ by 
\[X^D= e^{M_\Delta}(g)c^D \tilde c^D.\]
Then the Ricci operator is given by
\[\Ric = \frac12 (\tran{M_\Delta} X)^D.\]
\end{proposition}

\begin{theorem}
\label{thm:sigmacompatible}
Let $\g$ be a nice Lie algebra with diagram $\Delta$ and structure constants $c\in V_\Delta$. Let $\sigma$ be a diagram involution. Then $\g$ has a $\sigma$-diagonal metric $g$ such that $\logsign g=\delta$ and $\Ric=\frac12k\id$ if and only if for some $\sigma$-invariant $X\in \R^m$:
\begin{description}
\item[($\mathbf{K}$)] $\tran {M_\Delta} X=[k]$.
\item[($\mathbf{H}$)] $X$ does not belong to any coordinate hyperplane;
\item[($\mathbf{L}_\sigma$)\namedlabel{enum:condLsigma}{($\mathbf{L}_\sigma$)}] $\logsign X+ \logsign c+\logsign \tilde c =M_{\Delta,2}\delta$;
\item[($\mathbf{P}_\sigma$)\namedlabel{enum:condPsigma}{($\mathbf{P}_\sigma$)}] for a basis $\alpha_1,\dotsc, \alpha_k$ of $(\ker \tran{M_\Delta})^\sigma$, we have
\[\abs{X}^{\alpha_i}=\abs{c}^{2\alpha_i}, \quad i=1,\dotsc, k.\]
\end{description}
\end{theorem}
\begin{proof}
Following the proof of Theorem~\ref{thm:diagonal}, a $\sigma$-diagonal metric has the form $g=(-1)^\delta \exp v$, where both $\delta$ and $v$ are $\sigma$-invariant. The metric $g$  has signature $\delta$; it solves $\Ric=\frac12k\id$ if 
\[\tran{M_\Delta} X=[k], \quad X^D=(-1)^{M_{\Delta,2}(\delta)}\exp (M_\Delta(v)^D)c^D\tilde c^D;\]
equivalently, $X$ is characterized by
\[\logsign X=M_{\Delta,2}\delta+\logsign c+\logsign\tilde c, \quad \abs{X}^D=\exp (M_\Delta(v)^D)c^D\tilde c^D;\]
notice that $X$ is $\sigma$-invariant, because so are $c$, $\tilde c$ and $g$. Taking componentwise logarithms, we find
\begin{equation}
\label{eqn:XccTildav}
\log \abs{X}- \log\abs{c}-\log\abs{\tilde c} = M_\Delta(v).
\end{equation}
Thus, the existence of a $\sigma$-diagonal metric with signature $\delta$ and $\Ric=\frac12k\id$ is equivalent to the existence of a vector $X$ satisfying conditions~\ref{enum:condK},~\ref{enum:condH},~\ref{enum:condLsigma} and \eqref{eqn:XccTildav} for some $\sigma$-invariant $v\in d_n$. For fixed $X$, Equation~\eqref{eqn:XccTildav} has a solution in $v$ if and only if the left-hand side is orthogonal to $\ker \tran M_\Delta$; such a solution can always be assumed to be $\sigma$-invariant up to replacing $v$ with $\frac12(v+\sigma(v))$.

Since $\sigma$ is symmetric of order two, we have an orthogonal decomposition
\[\R^m=V_+\oplus V_-, \quad \sigma|_{V_\pm}=\pm\id|_{V_\pm}.\]
As $\ker \tran M_\Delta$ is $\sigma$-invariant, a vector in $V_+$ is orthogonal to $\ker \tran M_\Delta$ if and only if it is orthogonal to $\ker \tran M_\Delta\cap V_+ = (\ker \tran M_\Delta)^\sigma$.
Therefore, Equation~\eqref{eqn:XccTildav} has a solution in $v$ if and only if, for every $\alpha$ in $(\ker \tran M_\Delta)^\sigma$,
\[\alpha(\log \abs{X})
=\alpha(\log \abs{c})+\alpha(\log \abs{\tilde c})
=\alpha(\log \abs{c})+\alpha(\log \abs{\sigma(c)})
=2\alpha(\log \abs{c}),\]
where we have used $\abs{\tilde\sigma (c)}=\abs{\sigma(c)}$ and $\sigma$-invariance of $\alpha$. Last equation is equivalent to $\abs{X}^{\alpha}=\abs{c}^{2\alpha}$.
\end{proof}
\begin{example}
\label{example:2216}
To illustrate the case $k=0$, take the Lie algebra
\[ \texttt{731:15}\quad(0,0,0,0,e^{12},e^{13},e^{26}+e^{35}).\]
Then $\sigma=(23)(56)$ is an automorphism, acting on $V_\Delta\cong \R^4$ as  $(1 2)(3 4)$. The kernel of $\tran{M_\Delta}$ is generated by $(1,-1,-1,1)$, so it does not contain any non-trivial $\sigma$-invariant element. Therefore, this nice Lie algebra does not admit any $\sigma$-diagonal Ricci-flat metric.
\end{example}

Of special interest is the case of permutations $\sigma\in\Aut(\Delta)$ with no fixed points, corresponding to linear isomorphisms $\sigma\colon\R^n\to\R^n$ that do not fix any element in the nice basis. Then $n$ is even and a $\sigma$-diagonal metric has the form
\[g=\left(\begin{array}{cc|c|cc}
   0&1&& \\
   1& 0&& \\
   \hline
    & &  \ddots \\
    \hline
   &&& 0&1 \\   
   &&&1& 0 \\   
  \end{array}\right);\]
consequently, it has neutral signature. 

\begin{remark}\label{rem:ParaComplexSigmaNotIntegrable}
Recall that an almost paracomplex structure on a manifold $M$ is an endomorphism $K\colon TM\to TM$ such that $K^2=\id$ and the $\pm1$-eigendistributions have the same rank. A neutral metric $g$ is compatible with $K$ if $g(K\cdot,K\cdot)=-g$ (see \cite{RossiPhD} and the references therein); in this case $(K,g)$ is called an almost parahermitian structure.

A fixed-point-free permutation $\sigma\in\Aut(\Delta)$ defines an almost paracomplex structure which is not compatible in this sense with the $\sigma$-diagonal metric $g$, because the $\sigma$-invariant vectors are not null vectors for the metric $g$.  We observe that this almost paracomplex structure is not integrable (i.e. the eigendistributions are not involutive), unless the Lie algebra is abelian. Assume that $\g$ is not abelian, and suppose that $[e_i,e_j]=ae_k$ for some nonzero constant $a$. Since $\sigma$ is an automorphism, $[e_{\sigma_i},e_{\sigma_j}]=be_{\sigma_k}$ for some nonzero constant $b$. Write $f_i^\pm=e_i\pm e_{\sigma_i}$, so that $f_i^\pm$ is in the $\pm1$-eigenspace of $\sigma$. Then
\[
[f_i^+,f_j^+]=ae_k+be_{\sigma_k} + v,\qquad
[f_i^-,f_j^-]=ae_k+be_{\sigma_k} -v,
\]
where by the nice conditions $v$ lies in the span of the basis elements that differ from both $e_k$ and $e_{\sigma_k}$. Since the nonzero vector $ae_k+be_{\sigma_k}$ cannot be in both eigenspaces, at least one of the eigendistributions is not involutive. 
\end{remark}

\begin{remark}\label{rem:SigmaMetricParahermitian}
By contrast, given an automorphism $\sigma$ with no fixed point, it is possible to construct an almost parahermitian structure $(K,g)$ for any  $\sigma$-diagonal metric $g$. In fact, partition the nice basis as $\B=\B^+\cup \sigma(\B^+)$ and define an almost paracomplex structure $K$ such that $K|_{\mathcal{B}^+}=\id$ and $K|_{\sigma(\mathcal{B}^+)}=-\id$; the $\sigma$-diagonal metric is an almost parahermitian metric compatible with $K$.  We observe that, in general, the almost paracomplex structure is not integrable. These types of metrics fit into the framework of \cite{ContiRossi:ricci}; their Ricci tensor can therefore be related to the intrinsic torsion of the almost parahermitian structure  (see Example~\ref{example:ParaComplexRicciFlat}).
\end{remark}

\section{The case of nonzero scalar curvature}
\label{section:nonzero}
In this section we review our results concerning Einstein metrics of nonzero scalar curvature $s$ (namely, solutions of \eqref{eqn:einstein} with $\lambda\neq0$); the results stated here are contained in \cite{ContiRossi:Construction,ContiRossi:EinsteinNice,ContiRossi:EinsteinNilpotent}.

The existence of an Einstein metric with nonzero scalar curvature puts strong constraints on the Lie algebra, even outside of the nice context. In fact, the following holds:
\begin{theorem}[\cite{ContiRossi:EinsteinNilpotent}]
\label{th:OstruzioneEinstein}
Nilpotent Lie algebras admitting a derivation with nonzero trace do not carry Einstein metrics with $s\neq0$.
\end{theorem}
For nice Lie algebras, this obstruction only depends on the diagram, or equivalently the root matrix $M_\Delta$ (see \cite[Lemma 2.1]{ContiRossi:EinsteinNice}). Linear case-by-case computations lead to the following:
\begin{corollary}[\cite{ContiRossi:EinsteinNilpotent}]
\label{co:niceEinsteinDimBassa}
If $\g$ is either:
\begin{itemize}
\item a nilpotent Lie algebra of dimension $\leq 6$; or
\item a nice nilpotent Lie algebra of dimension $7$,
\end{itemize}
then $\g$ admits no Einstein metric with $s\neq0$. 
\end{corollary}
More precisely, from Theorem \ref{th:OstruzioneEinstein} it follows that at most $11$ nilpotent Lie algebras of dimension $7$ (none of which are nice) can admit an Einstein metric with $s\neq0$. This result makes it natural to ask:\footnote{After a first draft of this paper appeared online, an example providing a positive answer to this question was found in \cite{FernandezFreibertSanchez}. To our knowledge, the complete classification of nilpotent Lie algebras of dimension $7$ admitting an Einstein metric with $s\neq0$ remains an open problem.}

\begin{question}
Are there any nilpotent Lie algebras of dimension $7$ with an Einstein metric of nonzero scalar curvature?\end{question}

The difficulty of answering this question is that the Ricci tensor of a general $7$-dimensional metric has $\binom{8}{2}$ components, as opposed to the metrics considered in Section~\ref{sec:einsteinonnice}, that have at most $7$ nonzero components. Even if we restrict to diagonal or $\sigma$-diagonal metrics on nice Lie algebras, however, determining the existence of an Einstein metric generally requires solving nonlinear equations equivalent to \ref{enum:condP} and~\ref{enum:condPsigma}.

In dimension $8$, there are infinitely many inequivalent nice nilpotent Lie algebras; more precisely, there are $45$ continuous families and $872$ isolated nice Lie algebras  (see \cite{ContiRossi:Construction}). However, only few of them admit diagonal or $\sigma$-diagonal Einstein metrics with $s\neq0$:
\begin{theorem}[\protect{\cite[Theorem 4.4]{ContiRossi:EinsteinNice}}]
\label{thm:EinsteinMetrics8}
Among 8-dimensional nice nilpotent Lie algebras, up to equivalence:
\begin{itemize}
\item exactly 6 admit a diagonal Einstein metric with $s\neq0$;
\item exactly 4 admit a $\sigma$-diagonal Einstein metric with $s\neq0$ for some diagram involution $\sigma$.
\end{itemize}
\end{theorem}
\begin{remark}
The nice Lie algebras of Theorem~\ref{thm:EinsteinMetrics8} are listed in Table~\ref{table:8DimNiceEinsteinDiagonal}. We note that \texttt{842:117} and \texttt{842:121b} are isomorphic via
\begin{align*}
e^1&\mapsto e^1-e^4,&\quad e^2&\mapsto -e^2-e^3,&\quad e^3&\mapsto e^1-e^4,&\quad e^4&\mapsto e^2-e^3,\\
e^5&\mapsto -e^5-e^6,&\quad e^6&\mapsto e^5-e^6,&\quad e^7&\mapsto 2e^8,&\quad e^8&\mapsto 2e^7.
\end{align*}

With this exception, the Lie algebras of Table~\ref{table:8DimNiceEinsteinDiagonal} are pairwise nonisomorphic, as one can check by taking the quotient by a line in the center and using the classification of \cite{Gong}.
\end{remark}

\begin{table}[thp]
\centering
{\setlength{\tabcolsep}{2pt} 
\caption{8-dimensional nice Lie algebras with a non-Ricci-flat, Einstein metric of diagonal or $\sigma$-diagonal type\label{table:8DimNiceEinsteinDiagonal}}
\begin{tabular}{>{\ttfamily}c C C C}
\toprule
\textnormal{Name} & \g & \text{diag.}&\text{$\sigma$-comp.}\\
\midrule
86532:6 & \multicolumn{1}{L}{0,0,e^{12},-e^{13},e^{23},-e^{15}+e^{24},} & \checkmark & \checkmark  \\ 
 & \multicolumn{1}{R}{e^{16}+e^{25}+e^{34},e^{14}+e^{26}+e^{35}} & & \\[5pt]
8531:60a & 0,0,0,e^{12},e^{13},e^{24},e^{15}+e^{23},e^{14}+e^{26}+e^{35} & \checkmark&\\[5pt] 
8531:60b & 0,0,0,-e^{12},e^{13},e^{24},e^{15}+e^{23},e^{14}+e^{26}+e^{35} &\checkmark &\\[5pt] 
842:117  & 0,0,0,0,e^{12},e^{34},e^{15}+e^{24}+e^{36},e^{13}+e^{25}+e^{46}&\checkmark &\checkmark \\[5pt] 
842:121a & \multicolumn{1}{L}{0,0,0,0,e^{13}-e^{24},e^{12}+e^{34},}& \checkmark & \checkmark\\
 & \multicolumn{1}{R}{e^{14}+e^{25}+e^{36},e^{15}+e^{23}+e^{46}} & & \\[5pt]
842:121b & \multicolumn{1}{L}{0,0,0,0,-e^{13}+e^{24},-e^{12}+e^{34},} &\checkmark & \checkmark\\
 & \multicolumn{1}{R}{e^{14}+e^{25}+e^{36},e^{15}+e^{23}+e^{46}} & & \\
\bottomrule
\end{tabular}
}
\end{table}
It is  striking that the Lie algebras appearing in Table~\ref{table:8DimNiceEinsteinDiagonal} are precisely the $8$-dimensional nice Lie algebras that are \emph{characteristically nilpotent}; this condition means that all derivations are nilpotent (see \cite{AncocheaCampoamor} and the references therein). In particular, characteristically nilpotent Lie algebras trivially satisfy the obstruction of Theorem~\ref{th:OstruzioneEinstein}. In the nice context, the set of derivations diagonalized by the nice basis corresponds to the kernel of the root matrix; thus, the characteristically nilpotent condition implies that the root matrix is injective. However, the two conditions are not equivalent, as can be seen by considering the nice nilpotent Lie algebra
\[\texttt{9521:70a}\quad (0,0,0,0,e^{12},e^{14}+e^{23},e^{13}+e^{24},e^{15},e^{18}+e^{25}+e^{34}),\]
which has injective root matrix, but is not characteristically nilpotent. 

Notice that characteristically nilpotent Lie algebras only exist in dimension $7$ and greater (see \cite{Favre}); nice characteristically nilpotent Lie algebras, by contrast,
have dimension at least $8$, as one can verify using the classification. The above observations make it natural to ask:
\begin{question}
\label{question:characteristicallynilpotent}
Do all characteristically nilpotent Lie algebras admit an Einstein metric with $s\neq0$?
\end{question}

In dimension $9$ and higher, the polynomial equations become increasingly difficult to solve. However, we can use a simpler sufficient condition for the existence of an Einstein metric that only depends on linear computations:
\begin{theorem}[\protect{\cite[Theorem 2.9]{ContiRossi:EinsteinNice}}]\label{thm:Sufficient_Condition}
Let $\g$ be a nice nilpotent Lie algebra with diagram $\Delta$; if $M_{\Delta,2}$ is surjective and $X$ is a vector satisfying Conditions~\ref{enum:condK} and~\ref{enum:condH} in Theorem~\ref{thm:diagonal}, then there exists a diagonal Einstein metric with $s\neq0$.
\end{theorem}

Restricting to the case where $M_\Delta$ is surjective, we obtain the following:
\begin{theorem}[\protect{\cite[Theorem 4.8]{ContiRossi:EinsteinNice}}]
\label{thm:EinsteinMetrics9}
Among 9-dimensional nice nilpotent Lie algebras with surjective root matrix, up to equivalence:
\begin{itemize}
\item exactly 48 admit a diagonal Einstein metric with $s\neq0$;
\item exactly 7 admit a $\sigma$-diagonal  Einstein metric with $s\neq0$ for some diagram involution $\sigma$.
\end{itemize}
\end{theorem}

\begin{remark}
\label{remark:nilmanifold}
All the Lie algebras appearing in Theorems~\ref{thm:EinsteinMetrics8} and \ref{thm:EinsteinMetrics9} have rational structure constants. Therefore, the associated Lie group $G$ has a lattice $\Gamma$, and the left-invariant Einstein metric on $G$ induces an Einstein metric on a nilmanifold, namely the compact quotient $\Gamma\backslash G$.
\end{remark}
\begin{remark}
In light of Question~\ref{question:characteristicallynilpotent}, we notice that among the Einstein Lie algebras appearing in Theorem~\ref{thm:EinsteinMetrics9} there are $44$ characteristically nilpotent Lie algebras.
\end{remark}

In the situation of Theorem~\ref{thm:EinsteinMetrics9}, there do not appear families of nice nilpotent Lie algebras; this is a general consequence of the surjectivity of $M_{\Delta}$. It would be interesting to investigate the general behaviour of families of nice nilpotent Lie algebras sharing the same diagram and root matrix. Indeed, \cite{Nikolayevsky} showed that the existence of a Riemannian nilsoliton metric on a nice nilpotent Lie algebra only depends on the root matrix  --- or, in our language, the diagram. Since nilsoliton metrics, like Einstein metrics, arise as critical points of the scalar curvature, it is natural to ask whether this property applies to Einstein metrics in the pseudoriemannian case, i.e.:

\begin{question}
Are there any families of nice Lie algebras with the same diagram such that only some elements of the families admit a (diagonal) Einstein metric with $s\neq0$?
\end{question}

We point out that, with the methods of this section, it is not difficult to construct families of nice Lie algebras that have a diagonal Einstein metric for all values of the parameters (see e.g. \cite[Remark 4.6]{ContiRossi:EinsteinNice}).

In the case $s=0$, a family of nice Lie algebras can admit a Ricci-flat metric for all values of the parameters, some or none; among the Lie algebras considered in this paper, see e.g. \texttt{741:6} (Theorem~\ref{thm:7DiagRicciFlat}), \texttt{8542:15a} and \texttt{85321:48} (Theorem~\ref{thm:8dimricciflat}).

Fixing the Lie algebra, a different question to consider is the following:
\begin{question}
Are there any (nice) nilpotent Lie algebras with a family of non-isometric Einstein metrics with $s\neq0$?
\end{question}
The analogous  question for $s=0$ can be answered in the affirmative; see Example~\ref{example:familyRicciFlatMetrics}.

\section{Ricci-flat metrics}
\label{section:ricciflat}
In this section we apply Theorem~\ref{thm:diagonal} and Theorem~\ref{thm:sigmacompatible} constructively and classify nice Lie algebras of dimension $\leq8$ admitting diagonal or $\sigma$-diagonal Ricci-flat metrics, with $\sigma$ a diagram involution. Unlike the situation of Section~\ref{section:nonzero}, some of the Lie algebras obtained in this section appear in families; for rational values of the parameter(s), the argument of Remark~\ref{remark:nilmanifold} yields a Ricci-flat metric on a nilmanifold.

We use the classification of nice Lie algebras contained in \cite{ContiRossi:Construction}. We start by observing that, for $k=0$, conditions \ref{enum:condK} and \ref{enum:condH} can only be satisfied when $M_\Delta$ is nonsurjective. In dimension less than seven the only Lie algebra with nonsurjective $M_\Delta$ is the Lie algebra \texttt{631:6}. We will carry out the computations in detail in the following example, where we also introduce the conventions used throughout the section.

\begin{example}
\label{example:631}
Consider the Lie algebra:
\[\texttt{631:6}\quad (0,0,0,e^{12},e^{13},e^{25}+e^{34}).\]
Then $\tran{M_\Delta}X =0$ has general solution
\[X=(x,-x,-x,x),\]
and \ref{enum:condP} is trivially satisfied. If nonzero, the vector $X$ belongs to one of two orthants, according to the sign of $x$. By~\ref{enum:condL}, the vectors $\delta\in\Z_2^6$ such that $\logsign{X}\in M_{\Delta,2}(\delta)$ form the possible signatures of a diagonal Ricci-flat metric $g$. Here and in the sequel, we will refer to these vectors $\delta$ as the \emph{admissible signatures}.

Given a diagonal metric $g=g_i e^i\otimes e^i$ we will identify its signature by listing the indices $i$ for which $g_i$ is negative; for instance, $14$ stands for a diagonal metric such that
$g_1$ and $g_4$ are negative and the other $g_i$ are positive; in conventional language, the signature of this metric is $(n-2,2)$. The set of admissible signatures $\mathbf{S}$ will be ordered by length and lexicographic order. 

Recall that when $g$ is a Ricci-flat metric, then so is $-g$: thus, assuming the dimension is $6$, if $14$ is an admissible signature then $2356$ is also admissible. For brevity, we will only list one admissible signature in each like complementary pair, namely the one coming first in the order. On a fixed nice Lie algebra, we will denote by $\tfrac12\mathbf{S}$ the halved set of admissible signatures obtained in this way.

In particular, for \texttt{631:6} we obtain
\[\tfrac12\mathbf{S}=\{4,\ 5,\ 12,\ 13,\ 26,\ 36,\ 146,\ 156\};\]
this shows that this Lie algebra contains a Ricci-flat metric in each indefinite signature $(p,q)$.

To obtain the explicit metrics, we have to solve the system $e^{M_{\Delta}}(g)=X$, which (normalizing to $x=1$) leads to 
\[\frac{g_4}{g_1g_2} = -\frac{g_5}{g_1g_3} =- \frac{g_6}{g_2g_5} =\frac{g_6}{g_3g_4}=1. \]
This system has solution
\[g_4=g_1g_2,\quad g_5=-g_1g_3,\quad g_6=g_1g_2g_3,\]
giving a $3$-parameter family of Ricci-flat metrics, consistently with the fact that $\ker M_\Delta$ has dimension $3$ (see Remark~\ref{rem:ActionKerMDelta}).
In this case there are no diagram involutions $\sigma$ with a $\sigma$-diagonal metric because the only nontrivial automorphism is $(23)(45)$, which does not preserve any nonzero element of $\ker\tran M_\Delta$.
\end{example}

We have obtained the following:
\begin{proposition}
There are no diagonal Ricci-flat metrics on any nice nilpotent Lie algebra of dimension $\leq5$.

In dimension $6$, the nice nilpotent Lie algebra
\[\texttt{631:6}\quad (0,0,0,e^{12},e^{13},e^{25}+e^{34})\]
is the only one that admits a diagonal Ricci-flat metric, and the admissible signatures are represented by 
\[\tfrac12\mathbf{S}=\{4,\ 5,\ 12,\ 13,\ 26,\ 36,\ 146,\ 156\}.\]

There are no $\sigma$-diagonal Ricci-flat metrics on any nice nilpotent Lie algebra of dimension $\leq6$, for any diagram involution $\sigma$.
\end{proposition}

\begin{table}[thp]
\centering
{\setlength{\tabcolsep}{2pt} 
\caption{7-dimensional nice nilpotent Lie algebras that satisfy \ref{enum:condK},~\ref{enum:condH},~\ref{enum:condL} and~\ref{enum:condP} for $k=0$ \label{table:7amenable}}
\begin{tabular}{>{\ttfamily}c C}
\toprule
\textnormal{Name} & \g \\
\midrule
75432:3 & 0,0,- e^{12},e^{13},e^{14}+e^{23},e^{15}+e^{24},e^{25}+e^{34}\\ 
\multicolumn{2}{C}{\tfrac12\mathbf{S}=\{5,47,137,267\}}\\[8pt]
741:6 & 0,0,0, {(\lambda-1)} e^{12}, \lambda e^{13},e^{23},e^{16}+e^{25}+e^{34}\\
\multicolumn{2}{C}{\lambda>1\quad\tfrac12\mathbf{S}=\{ 6,17,23,45,125,134,247,357\}}\\
\multicolumn{2}{C}{0<\lambda<1\quad\tfrac12\mathbf{S}=\{5,13,27,46,126,147,234,367\}}\\
\multicolumn{2}{C}{\lambda<0\quad\tfrac12\mathbf{S}=\{4,12,37,56,136,157,235,267\}}\\[8pt]
731:15&0,0,0,0,e^{12},e^{13},e^{26}+e^{35}\\
\multicolumn{2}{C}{\tfrac12\mathbf{S}=\{5,6,12,13,27,37,45,46,124,134,157,167,235,247,236,347\}}\\
\bottomrule
\end{tabular}
}
\end{table}

In dimension $7$, we obtain the following:

\begin{theorem}\label{thm:7DiagRicciFlat}
The $7$-dimensional nice nilpotent Lie algebras admitting a diagonal Ricci-flat metric are listed in Table~\ref{table:7amenable}.
\end{theorem}
\begin{proof}
Linear computations on a case-by-case basis show that the nice Lie algebras of dimension $7$ that satisfy \ref{enum:condK} and~\ref{enum:condH} are precisely those listed in Table~\ref{table:7amenable} plus the following:
\begin{align*} 
\texttt{754321:9}&\quad (0,0, {(1-\lambda)}e^{12} ,e^{13}, \lambda e^{14}+e^{23},e^{15}+e^{24},e^{16}+e^{25}+e^{34}),\\
\texttt{75421:4}&\quad (0,0,e^{12},e^{13},e^{23},e^{15}+e^{24},e^{16}+e^{34}),\\
\texttt{74321:12}&\quad (0,0,0,- e^{12},e^{14}+e^{23},e^{15}+e^{34},e^{16}+e^{35}).
\end{align*}
The computations for \texttt{754321:9} have been made in Example~\ref{es:754321:9}.

For \texttt{75421:4}  we find that vectors in $\ker \tran{M_\Delta}$ have the form  $X=(x_7,x_7+x_6,-x_7-x_6,-x_7-x_6,x_6,-x_7,x_7)$, and 
\ref{enum:condP} reads
\[\left|\frac{x_6}{x_6+x_7}\right|=1,\quad\left|\frac{x_7}{x_6+x_7}\right|=1;\]
the only solution is $x_6=x_7=0$, which violates \ref{enum:condH}; similarly for \texttt{74321:12}. 

For \texttt{75432:3}, \ref{enum:condP} is trivially satisfied. Imposing \ref{enum:condL} on a generic vector of $\ker\tran M_\Delta$ not contained in a coordinate hyperplane, we find $8$ admissible signatures, i.e. 
\[\tfrac12\mathbf{S}=\{5,\ 47,\ 137,\ 267\};\]
see also Example~\ref{example:familyRicciFlatMetrics}.
Similarly for \texttt{731:15}; the admissible signatures for this nice Lie algebra are listed in Table~\ref{table:7amenable}. 

For \texttt{741:6}, we get $X=(x_6,-x_5-x_6,x_5,x_5,-x_5-x_6,x_6)$, and \ref{enum:condP} gives: 
\[\frac{x_6^{2}}{{(x_6+x_5)}^{2}}={(-1+\lambda)}^{2},\quad \frac{x_5^{2}}{{(x_6+x_5)}^{2}}=\lambda^{2},\]
which has a solution $x_6= (\frac{1}{\lambda}-1)x_5$ for all $\lambda$. Note that $x_5+x_6$ has the same sign as $\lambda$, and $x_5$ is positive if and only if $x_6$ has the same sign as $-1+1/\lambda$. Hence, the admissible signatures are represented by:
\begin{align*}
\lambda>1:\quad\tfrac12\mathbf{S}&=\{ 6,\ 17,\ 23,\ 45,\ 125,\ 247,\ 134,\ 357\}\\
0<\lambda<1:\quad\tfrac12\mathbf{S}&=\{5,\ 13,\ 27,\ 46,\ 126,\ 147,\ 234,\ 367\}\\
\lambda<0:\quad\tfrac12\mathbf{S}&=\{4,\ 37,\ 12,\ 56,\ 136,\ 157,\ 235,\ 267\}.
\end{align*}
These results are collected in Table~\ref{table:7amenable}.
\end{proof}

For a $\sigma$-diagonal metric $g=g_{i}e^i\otimes e^{\sigma_i}$, we will also denote by $g_{ij}$ the coefficient of $e^i\otimes e^j$,  and represent the signature by listing the indices $i$ for which  $g_{i\sigma_i}=g_i$ is negative. We remark that if $\sigma$ interchanges $i$ and $j$, then $g_{i\sigma_i}$ and $g_{j\sigma_j}$ coincide by construction, so either both or none of $i$ and $j$ must appear in the list. For each signature $(p,q)$,  we will denote by $\mathbf{S}_\sigma(p,q)$ the set of $\delta\in\Z_2^n$ such that there is a diagram involution $\sigma$ and a Ricci-flat $\sigma$-diagonal metric $g$ of signature $(p,q)$ with $\logsign g=\delta$; as usual, each $\delta$ will be represented by the indices of its nonzero entries.

With the aim of constructing $\sigma$-diagonal Ricci-flat metrics, we apply Theorem~\ref{thm:sigmacompatible} to nice nilpotent Lie algebras. Since conditions~\ref{enum:condK},~\ref{enum:condH} are the same as in Theorem~\ref{thm:diagonal}, we restrict our attention to those Lie algebras that satisfy both conditions and that have an order $2$ automorphism of the diagram $\Delta$. We obtain:

\begin{theorem}
\label{thm:sigmacompatible7}
The only nonabelian nice nilpotent Lie algebra of dimension $7$ with a diagram involution $\sigma$ and a Ricci-flat $\sigma$-diagonal metric is 
\[\texttt{741:6}\quad(0,0,0, (\lambda - 1) e^{12}, \lambda e^{13},e^{23},e^{16}+e^{25}+e^{34}) \]
with \begin{itemize}
	\item $\sigma=(23) (45)$, $\lambda=\dfrac12$ and  \[\mathbf{S}_\sigma(4,3)=\{1,\ 237,\ 457,\ 12345\},\quad \mathbf{S}_\sigma(3,4)=\{67,\ 1236,\ 1456,\ 234567\};\]
	\item $\sigma=(12)(56)$, $\lambda=-1$ and  \[\mathbf{S}_\sigma(4,3)=\{3,\ 127,\ 567,\ 12356\}, \quad \mathbf{S}_\sigma(3,4)=\{47,\ 1234,\ 3456,\ 124567\};\]
	\item $\sigma=(13)(46)$, $\lambda=2$ and \[\mathbf{S}_\sigma(4,3)=\{2,\ 137,\ 467,\ 12346\},\quad \mathbf{S}_\sigma(3,4)=\{57,\ 1235,\ 2456,\ 134567\}.\]
\end{itemize}
\end{theorem}
\begin{proof}
By Theorem~\ref{thm:sigmacompatible}, a $\sigma$-diagonal metric exists only if there exists a $\sigma$-invariant $X\in\ker \tran M_\Delta$ which is not contained in a coordinate hyperplane. This eliminates all $7$-dimensional nice Lie algebras except \texttt{741:6}. We have three possibilities for $\sigma\in\Aut(\Delta)$:

$\bullet\ \sigma=(23)(45)$, acting on $\R^6\cong V_\Delta$ as $(12) (56)$. The elements in $(\ker \tran M_\Delta)^\sigma$ can be written as
\[X=(x,x,-2x,-2x,x,x),\]
whilst $c=(\lambda-1, \lambda,1,1,1,1)$ and $\tilde c=(\lambda,\lambda-1,-1,1,1,1)$. Condition~\ref{enum:condPsigma} reads
\[
(\lambda-1)^2\lambda^2=\frac{x^4}{(2x)^4}=\frac1{16};\]
condition~\ref{enum:condLsigma} holds if
\[\epsilon=(\logsign(\lambda-1)+\logsign \lambda,\logsign(\lambda-1)+\logsign \lambda,1,0,0,0)+\logsign X\]
is in $\im M_{\Delta,2}$. This condition is unaffected by changing the sign of $X$, since $[1]$ belongs to $\im M_{\Delta,2}$; in fact, changing the sign of $X$ corresponds to changing the sign of the corresponding $\sigma$-diagonal metrics, which clearly preserves Ricci-flatness.

Thus, we can assume
\[\epsilon = (\logsign(\lambda-1)+\logsign \lambda,\logsign(\lambda-1)+\logsign \lambda,0,1,0,0);\]
then $\epsilon$ is in the image of $M_{\Delta,2}$ 
 for $\logsign(\lambda-1)+\logsign \lambda =1$, i.e.
$\lambda(\lambda-1)<0$. This implies $\lambda=\dfrac12$ and $\epsilon=(1,1,0,1,0,0)$. The $\sigma$-invariant solutions of  $M_{\Delta,2}\delta=\epsilon$ are
\[(1,0,0,0,0,0,0),(1,1,1,1,1,0,0),(0,0,0,1,1,0,1),(0,1,1,0,0,0,1),\]
which we abbreviate as $1$, $12345$, $457$ and $237$. The actual signature of the resulting $\sigma$-diagonal metrics is $(4,3)$, as one can see by writing 
\[g=g_{11}e^1\otimes e^1+g_{23}e^2\odot e^3+g_{45}e^4\odot e^5+g_{66}e^6\otimes e^6+g_{77}e^7\otimes e^7.\] 
The other possible signature (corresponding to changing the sign of $X$) is $(3,4)$; summing up,
\[
\mathbf{S}_\sigma(4,3)=\{1,\ 237,\ 457,\ 12345\},\quad\mathbf{S}_\sigma(3,4)=\{67,\ 1236,\ 1456,\ 234567\}.
\]

$\bullet\ \sigma=(12)(56)$, acting on $\R^6$ as $(23) (45)$, with
$X=(2  x,- x,- x,- x,- x,2  x)$, $c=(\lambda-1, \lambda,1,1,1,1)$ and $\tilde c=(1-\lambda,1,\lambda,1,1,1)$.
In this case we get \[\frac{(\lambda-1)^4}{\lambda^2}=\frac{(2x)^4}{x^4}=16.\]
For $x>0$, \ref{enum:condLsigma} reads
\[(0,\logsign \lambda,\logsign \lambda,0,0,1)=M_{\Delta,2}(\delta),\]
which has a solution in $\delta$ if and only if $\lambda<0$. Therefore, there only exist Ricci-flat $\sigma$-diagonal metrics for $\lambda=-1$. Computations as in the previous case show that the possible signatures are
\[
\mathbf{S}_\sigma(4,3)=\{3,\ 127,\ 567,\ 12356\},\quad\mathbf{S}_\sigma(3,4)=\{47,\ 1234,\ 3456,\ 124567\}.
\]

$\bullet\ \sigma=(13) (46)$, acting on $\R^6$ as $(13) (46)$, with $X=( - x,2  x,- x,- x,2  x,- x)$,
$c=(\lambda-1,\lambda,1,1,1,1)$ and $\tilde c=(-1,-\lambda,1-\lambda,1,1,1)$. We get
\[\frac{\lambda^4}{(\lambda-1)^2} =16\]
and, for $x>0$,
$\epsilon =(\logsign(\lambda-1),1,\logsign (\lambda-1),1,0,1),$
giving $\lambda-1>0$, i.e. $\lambda=2$. The possible signatures are listed below:
\[
\mathbf{S}_\sigma(4,3)=\{2,\ 137,\ 467,\ 12346\},\quad\mathbf{S}_\sigma(3,4)=\{57,\ 1235,\ 2456,\ 134567\}.\qedhere
\]
\end{proof}

\FloatBarrier
\begin{footnotesize}
{\setlength{\tabcolsep}{2pt}
\begin{longtable}[c]{>{\ttfamily}c C}
\caption{8-dimensional nice nilpotent Lie algebras that satisfy \ref{enum:condK},~\ref{enum:condH},~\ref{enum:condL} and~\ref{enum:condP} for $k=0$ \label{table:8amenable}}\\
\toprule
\textnormal{Name} & \g \\
\midrule
\endfirsthead
\multicolumn{2}{c}{\tablename\ \thetable\ -- \textit{Continued from previous page}} \\
\toprule
\textnormal{Name} & \g \\
\midrule
\endhead
\bottomrule\\[-7pt]
\multicolumn{2}{c}{\tablename\ \thetable\ -- \textit{Continued to next page}} \\
\endfoot
\bottomrule\\[-7pt]
\endlastfoot
86531:5&0,0,- e^{12},e^{13},e^{23},e^{14},e^{15}+e^{24},e^{26}+e^{34}\\* 
\multicolumn{2}{C}{\tfrac12\mathbf{S}=\{48,56,125,138,278,357,1457,1678\}}\\[8pt]
85432:9&0,0,0,- e^{12},e^{14},e^{15}+e^{24},e^{16}+e^{25},e^{26}+e^{45}\\* 
\multicolumn{2}{C}{\tfrac12\mathbf{S}=\{6,36,58,148,278,358,1348,1456\}}\\[8pt]
8542:11&0,0,0,e^{12},e^{13}+e^{24},e^{14},e^{25}+e^{34},e^{15}+e^{26}\\* 
\multicolumn{2}{C}{\tfrac12\mathbf{S}=\{5,126,138,147,234,278,367,468\}}\\[8pt]
8542:15a&0,0,0,{(a_2-1)} e^{12} , a_2 e^{13}+e^{24},e^{14}+e^{23},e^{15}+e^{26},e^{16}+e^{25}+e^{34}\\*
\multicolumn{2}{C}{a_2>1\quad\tfrac12\mathbf{S}=\{5,126,137,148,234,278,368,467\}}\\*
\multicolumn{2}{C}{0<a_2<1\quad\tfrac12\mathbf{S}=\{6,125,134,178,237,248,358,457\}}\\[8pt]
8542:15b&0,0,0,{(a_2-1)} e^{12},a_2 e^{13}+e^{24},-e^{14}+e^{23},-e^{15}+e^{26},e^{16}+e^{25}+e^{34}\\*
\multicolumn{2}{C}{a_2>1\quad\tfrac12\mathbf{S}=\{5,126,137,148,234,278,368,467\}}\\*
\multicolumn{2}{C}{0<a_2<1\quad\tfrac12\mathbf{S}=\{6,125,134,178,237,248,358,457\}}\\[8pt]
842:111a&0,0,0,0,e^{14}+e^{23},e^{13}+e^{24},e^{16}+e^{25},e^{15}+e^{26}\\*
\multicolumn{2}{C}{\tfrac12\mathbf{S}=\{5,6,125,126,134,137,148,178,234,237,248,278,358,368,457,467\}}\\[8pt]
842:111b&0,0,0,0,-e^{14}+e^{23},e^{13}+e^{24},-e^{16}+e^{25},e^{15}+e^{26}\\*
\multicolumn{2}{C}{\tfrac12\mathbf{S}=\{5,6,125,126,134,137,148,178,234,237,248,278,358,368,457,467\}}\\[8pt]
841:48&0,0,0,0, (a_2-1) e^{12} , a_2 e^{13},e^{23},e^{17}+e^{26}+e^{35}\\*
\multicolumn{2}{C}{a_2>1\quad\tfrac12\mathbf{S}=\{6,13,28,46,57,127,134,158,235,248,378,457,1247,1256,1458,1678\}}\\*
\multicolumn{2}{C}{0<a_2<1\quad\tfrac12\mathbf{S}=\{7,18,23,47,56,126,135,148,234,258,368,456,1246,1345,1367,1257\}}\\*
\multicolumn{2}{C}{a_2<0\quad\tfrac12\mathbf{S}=\{5,12,38,45,67,124,137,168,236,278,348,467,1347,1356,1468,1578\}}\\[8pt]
831:37&0,0,0,0,0,e^{12},e^{13},e^{27}+e^{36}\\*
\multicolumn{2}{L}{\tfrac12\mathbf{S}=\{6,7,12,13,28,38,46,47,56,57,124,125,134,135,168,178,236,237,248,}\\*
\multicolumn{2}{R}{258,348,358,457,456,1245,1267,1345,1367,1468,1478,1568,1578\}}\\
\end{longtable}
}
\end{footnotesize}

We now turn our attention to nice nilpotent Lie algebras of dimension $8$.

\begin{theorem}
\label{thm:8dimricciflat}
The $8$-dimensional nice nilpotent Lie algebras admitting a diagonal Ricci-flat metric are listed in Table~\ref{table:8amenable}.
\end{theorem}
\begin{proof}
For each $8$-dimensional nice nilpotent Lie algebra in the classification of~\cite{ContiRossi:Construction}, we compute the generic vector $X$ in $\ker \tran M_\Delta$, ruling out those Lie algebras for which no choices of $X$ satisfies \ref{enum:condK} and~\ref{enum:condH}. In the remaining cases we need to study the polynomial condition~\ref{enum:condP} and its compatibility with \ref{enum:condH} and~\ref{enum:condL}. We omit explicit computations, since they are essentially the same as in Example~\ref{example:631} and Theorem~\ref{thm:7DiagRicciFlat}, although longer and more tedious, and only summarize the results. The nice Lie algebras:\begin{footnotesize}
\begin{align*}
\texttt{8654321:19}&\ (0,0,e^{12},e^{13},e^{14}+e^{23},e^{15}+e^{24},2 e^{16}+e^{25}+e^{34},e^{17}+e^{35})\\
\texttt{8654321:20}&\ (0,0,e^{12},- e^{13},e^{14}+e^{23},-e^{15}+e^{24},e^{16}+e^{34},e^{17}+e^{26}+e^{35})\\
\texttt{8654321:25}&\ (0,0,- e^{12},- e^{13},\frac{3}{2} e^{14}-\frac{1}{2} e^{23},e^{15}+\frac{1}{2} e^{24},e^{16}+e^{25}+e^{34},e^{27}+e^{36}+e^{45})\\
\texttt{865431:9}&\ (0,0,- e^{12},2 e^{13},e^{14}+e^{23},2 e^{15}+e^{24},e^{25}+e^{34},e^{17}+e^{26}+e^{35})\\
\texttt{854321:25}&\ (0,0,0,- e^{12},e^{14}+e^{23},e^{15}+e^{34},e^{16}+e^{24}+e^{35},e^{17}+e^{25}+e^{36})\\
\texttt{85421:26}&\ (0,0,0,e^{12},e^{14},e^{24},e^{16}+e^{25},e^{17}+e^{45})\\
\texttt{8531:90}&\ (0,0,0,e^{12},e^{13},e^{24},e^{25}+e^{34},e^{27}+e^{36})\\
\texttt{8531:93}&\ (0,0,0,- e^{12},e^{13},e^{14},e^{15},e^{27}+e^{36}+e^{45})\\
\texttt{84321:20}&\ (0,0,0,0,- e^{12},e^{15}+e^{23},e^{16}+e^{35},e^{17}+e^{36})\\
\texttt{842:107}&\ (0,0,0,0,- e^{12},-e^{14}+e^{23},e^{16}+e^{35},e^{26}+e^{45})
\end{align*}
\end{footnotesize}do not satisfy \ref{enum:condP}. The following nice Lie algebras do not admit a vector $X$ satisfying all of \ref{enum:condH},~\ref{enum:condL} and~\ref{enum:condP}:\begin{footnotesize}
\begin{align*}
\texttt{8654321:24}&\ (0,0,{(1-a_4)}e^{12} ,a_4 e^{13},e^{14}+e^{23},\\
&\qquad\qquad\qquad\qquad e^{15}+a_4e^{24},{(2-a_4)} e^{16}+e^{25}+e^{34},e^{17}+e^{26}+e^{35})\\
\texttt{86531:14}&\ (0,0,a_1e^{12},e^{13},e^{23},(1-a_1) e^{14} ,e^{15}+e^{24},e^{17}+e^{26}+e^{34})\\
\texttt{854321:20}&\ (0,0,0,{(1-a_3)}e^{12} ,e^{14}, a_3 e^{15}+e^{24},e^{16}+e^{25},e^{17}+e^{26}+e^{45})\\
\texttt{85321:48}&\ (0,0,0,e^{12}, (a_3-1)e^{13} , a_3 e^{14}+e^{23} ,e^{25}+e^{34}+ {(a_3-2)} e^{16},e^{36}+e^{17}+e^{45})\\
\texttt{8531:103}&\ (0,0,0,e^{12},e^{13},2  e^{14},e^{25}+e^{34},e^{17}+e^{36}+e^{45})\\
\texttt{8521:12}&\ (0,0,0,e^{12},- e^{13},e^{23},-2 e^{16}+e^{25}+e^{34},e^{17}+e^{45}).
\end{align*}
\end{footnotesize}The same holds for \texttt{8542:15a} and \texttt{8542:15b} with $a_2<0$ (recall that parameters that appear in the structure constants of a family of nice Lie algebras are always assumed to be nonzero).

The nice nilpotent Lie algebras that satisfy all the conditions of Theorem~\ref{thm:diagonal} are listed in  Table~\ref{table:8amenable} together with their admissible signatures, according with the convention explained in Example~\ref{example:631}.
\end{proof}

Notice that some of the Lie algebras listed in Tables~\ref{table:7amenable} and~\ref{table:8amenable} are decomposable in the category of nice Lie algebras: since a diagonal metric is the product of diagonal metrics on each factor, the Ricci-flat diagonal metric can be recovered from Ricci-flat diagonal metrics of lower dimension. In particular we have that:
\begin{align*}
\texttt{731:15}&= \texttt{631:6}\oplus \R\\
\texttt{85432:9}&= \texttt{75432:3}\oplus \R\\
\texttt{841:48}&= \texttt{741:6}\oplus \R\\
\texttt{831:37}&= \texttt{731:15}\oplus \R=\texttt{631:6}\oplus \R^2
\end{align*}
and the admissible signatures of diagonal Ricci-flat metrics on these Lie algebras can be recovered from the admissible signatures in lower dimension.

\FloatBarrier
\begin{footnotesize}
{\setlength{\tabcolsep}{2pt}
\begin{longtable}[c]{>{\ttfamily}c C}
\caption{8-dimensional nice nilpotent Lie algebras that admit an order two automorphism $\sigma$ and that satisfy \ref{enum:condK},~\ref{enum:condH},~\ref{enum:condLsigma} and~\ref{enum:condPsigma} for $k=0$\label{table:8amenableSigma}}\\
\toprule
\textnormal{Name} & \g\\
$\sigma$ & \text{Metric} \\
\midrule
\endfirsthead
\multicolumn{2}{c}{\tablename\ \thetable\ -- \textit{Continued from previous page}} \\
\toprule
\textnormal{Name} & \g\\
$\sigma$ & \text{Metric} \\
\midrule
\endhead
\bottomrule\\[-7pt]
\multicolumn{2}{c}{\tablename\ \thetable\ -- \textit{Continued to next page}} \\
\endfoot
\bottomrule\\[-7pt]
\endlastfoot
852:30&0,0,0, a_1 e^{12}, a_2 e^{13},e^{23},e^{14}+e^{26}+e^{35},e^{15}+e^{24}+e^{36}\\*
\textnormal{(23)(45)(78)}& a_2 a_1=2\quad (g_{11},g_{23},g_{45},g_{66},g_{78})\\*
& \mathbf{S}_{\sigma}(5,3)=\{2345\},\quad \mathbf{S}_{\sigma}(4,4)=\{236,14578\},\quad \mathbf{S}_{\sigma}(3,5)=\{1678\}\\[3pt]
\textnormal{(12)(56)(78)}&a_2=2 a_1^2\quad (g_{12},g_{33},g_{44},g_{56},g_{78})\\*
&\mathbf{S}_{\sigma}(5,3)=\{1256\},\quad\mathbf{S}_{\sigma}(4,4)=\{124,35678\},\quad\mathbf{S}_{\sigma}(3,5)=\{3478\} \\[3pt]
\textnormal{(13)(46)(78)}&a_1=-2 a_2^2\quad (g_{13},g_{22},g_{46},g_{55},g_{78})\\*
&\mathbf{S}_{\sigma}(5,3)=\{1346\},\quad\mathbf{S}_{\sigma}(4,4)=\{135,24678\},\quad\mathbf{S}_{\sigma}(3,5)=\{2578\}\\[8pt]
842:111a&0,0,0,0,e^{14}+e^{23},e^{13}+e^{24},e^{16}+e^{25},e^{15}+e^{26}\\*
\textnormal{(34)(56)(78)}& (g_{11},g_{22},g_{34},g_{56},g_{78})\\*
&\mathbf{S}_{\sigma}(4,4)=\{134,178,234,278,13456,15678,23456,25678\}\\[3pt]
\textnormal{(12)(34)(78)}& (g_{12},g_{34},g_{55},g_{66},g_{78})\\*
&\mathbf{S}_{\sigma}(4,4)=\{5,6,125,126,34578,34678,1234578,1234678\}\\[8pt]
841:48&0,0,0,0, {(-1+a_2)} e^{12}, a_2 e^{13},e^{23},e^{17}+e^{26}+e^{35}\\*
\textnormal{(23)(56)}&a_2=\frac{1}{2}\quad(g_{11},g_{23},g_{44},g_{56},g_{77},g_{88})\\*
&\mathbf{S}_{\sigma}(5,3)=\{1,238,568,12356\}\\*
&\mathbf{S}_{\sigma}(4,4)=\{14,78,1237,1567,2348,4568,123456,235678\}\\*
&\mathbf{S}_{\sigma}(3,5)=\{478,12347,14567,2345678\}\\[3pt]
\textnormal{(12)(67)}&a_2=-1\quad(g_{12},g_{33},g_{44},g_{55},g_{67},g_{88})\\*
&\mathbf{S}_{\sigma}(5,3)=\{3,128,678,12367\}\\*
&\mathbf{S}_{\sigma}(4,4)=\{34,58,1235,1248,3567,4678,12348,125678\}\\*
&\mathbf{S}_{\sigma}(3,5)=\{458,12345,34567,1245678\}\\[3pt]
\textnormal{(13)(57)}&a_2=2\quad(g_{13},g_{22},g_{44},g_{57},g_{66},g_{88})\\*
&\mathbf{S}_{\sigma}(5,3)=\{2,138,567,12357\}\\*
&\mathbf{S}_{\sigma}(4,4)=\{24,68,1236,1348,2567,4578,123457,135678\}\\*
&\mathbf{S}_{\sigma}(3,5)=\{468,12348,24567,1345678\}\\[8pt]
831:37&0,0,0,0,0,e^{12},e^{13},e^{27}+e^{36}\\*
\textnormal{(45)}&(g_{11},g_{22},g_{33},g_{45},g_{66},g_{77},g_{88})\\*
&\mathbf{S}_{\sigma}(6,2)=\{6,7,456,457\}\\*
&\mathbf{S}_{\sigma}(5,3)=\{12,13,28,38,1245,1345,2458,3458\}\\*
&\mathbf{S}_{\sigma}(4,4)=\{168,178,236,237,14568,14578,23456,23457\}\\*
&\mathbf{S}_{\sigma}(3,5)=\{1267,1367,2678,3678,124567,134567,245678,345678\}\\*
&\mathbf{S}_{\sigma}(2,6)=\{12368,12378,1234568,1234578\}\\
\end{longtable}
}
\end{footnotesize}

Finally, we apply Theorem~\ref{thm:sigmacompatible} in the $8$-dimensional case, obtaining the following classification. 
\begin{theorem}
The $8$-dimensional nice nilpotent Lie algebras admitting a diagram involution $\sigma$ and a $\sigma$-diagonal diagonal Ricci-flat metric are listed in Table~\ref{table:8amenableSigma}.
\end{theorem}
\begin{proof}
As in the proof of Theorem~\ref{thm:sigmacompatible7}, we first list all the nice Lie algebras that admit a nontrivial automorphism $\sigma$ of order two
such that~\ref{enum:condK} has a $\sigma$-invariant solution that does not belong to any coordinate hyperplane.

Among the surviving Lie algebras, we find that the following:
\begin{align*}
\texttt{8531:93}&\quad (0,0,0,- e^{12},e^{13},e^{14},e^{15},e^{27}+e^{36}+e^{45})\\
\texttt{842:107}&\quad (0,0,0,0,- e^{12},-e^{14}+e^{23},e^{16}+e^{35},e^{26}+e^{45})
\end{align*}
do not satisfy \ref{enum:condPsigma}. The following list contains the Lie algebras that do not satisfy all of \ref{enum:condH},~\ref{enum:condLsigma} and~\ref{enum:condPsigma}:
\begin{align*}
\texttt{8521:12}&\quad (0,0,0,e^{12},- e^{13},e^{23},-2 e^{16}+e^{25}+e^{34},e^{17}+e^{45})\\
\texttt{8431:30}&\quad (0,0,0,0, a_1 e^{12},e^{15}+e^{23},e^{14}+e^{25},e^{17}+e^{26}+e^{34})\\
\texttt{842:111b}&\quad (0,0,0,0,-e^{14}+e^{23},e^{13}+e^{24},-e^{16}+e^{25},e^{15}+e^{26}).
\end{align*}
The same applies to \texttt{852:30} and \texttt{841:48} for values of the parameter different from those appearing in Table~\ref{table:8amenableSigma}.

For the remaining Lie algebras and each automorphism $\sigma$ of order $2$, we impose the conditions~\ref{enum:condK},~\ref{enum:condH},~\ref{enum:condLsigma} and~\ref{enum:condPsigma}. The results are listed in Table~\ref{table:8amenableSigma}.
\end{proof}
The above classifications show that diagonal and $\sigma$-diagonal metrics Ricci-flat metrics are quite scarce in the nice nilpotent context, although most of these Lie algebras do admit a Ricci-flat metric of more general type (see the references quoted in the introduction, as well as the forthcoming \cite{delBarcoContiRossi:SigmaType}). For example, all the nearly parak\"ahler $8$-dimensional examples presented in \cite[Section 7]{ContiRossi:ricci} admit a nice basis; in the notation of \cite{ContiRossi:Construction}, they can be written as
\begin{align*}
\texttt{8431:10}&\quad (0,0,0,e^{12},e^{14}+e^{23},e^{24},e^{15}-e^{34})\\
\texttt{8431:12a}&\quad (0,0,0,e^{12},e^{23}+e^{14},e^{13}+e^{24},e^{15}-e^{34})\\
\texttt{8431:26a}&\quad (0,0,e^{12},0,e^{13}+e^{24},e^{23},e^{15}+e^{26}+e^{34})\\
\texttt{8431:12b}&\quad (0,0,0,e^{12},e^{14}+e^{23},e^{13}-e^{24},e^{15}-e^{34})\\
\texttt{8431:26a}&\quad (0,0,e^{12},0,e^{13}+e^{24},e^{23},e^{15}-e^{26}+e^{34}).
\end{align*}
Thus, these nice Lie algebras admit a Ricci-flat metric of neutral signature, but not a diagonal or $\sigma$-diagonal Ricci-flat metric with $\sigma$ a diagram involution, as they do not appear in Tables~\ref{table:8amenable} and~\ref{table:8amenableSigma}. This is in sharp contrast with the case of nonzero scalar curvature, where the only known examples of Einstein nilpotent Lie algebras are nice Lie algebras with a diagonal or $\sigma$-diagonal metric.

We end this section with more examples of Ricci-flat nilpotent Lie algebras obtained by using Theorems~\ref{thm:diagonal} and~\ref{thm:sigmacompatible}.

\begin{example}
\label{example:familyRicciFlatMetrics}
Consider the nice Lie algebra $\g$:
\[\texttt{75432:3}\quad (0,0,- e^{12},e^{13},e^{14}+e^{23},e^{15}+e^{24},e^{25}+e^{34}).\]
We see that the generic vector $X\in\ker \tran M_\Delta$ is given by 
\[X=(x,y,-x,-y,-y,y,-x,x),\] 
and for $x,y\neq0$ it satisfies \ref{enum:condK},~\ref{enum:condH},~\ref{enum:condL} and~\ref{enum:condP}. The equation $e^{M_\Delta}(g)=X$ reads
\begin{gather*}
\frac{g_3}{g_1g_2}=x,\qquad\frac{g_4}{g_1g_3}=y,\qquad\frac{g_5}{g_1g_4}=-x,\qquad\frac{g_5}{g_2g_3}=-y,\\
\frac{g_6}{g_1g_5}=-y,\qquad\frac{g_6}{g_2g_4}=y,\qquad\frac{g_7}{g_2g_5}=-x,\qquad\frac{g_7}{g_3g_4}=x,
\end{gather*}
which has the following solution:
\begin{gather*}
g_2=x g_1^2,\qquad g_3=x^2 g_1^3,\qquad g_4=x^2 y g_1^4,\\
g_5=-x^3 y g_1^5,\qquad g_6=x^3 y^2 g_1^6,\qquad g_7=x^5 y g_1^7.
\end{gather*}
Thus, we obtain a family of Ricci-flat metrics depending on $3$ parameters. We can rescale the metric to normalize $\ad|_{\g'}\colon \g'\to \g'$, imposing
\[1=g(\ad|_{\g'},\ad|_{\g'})=g(e^{34}\otimes e_7,e^{34}\otimes e_7)=x;\]
in addition, the parameter $g_1$ can be eliminated since it reflects the kernel of $M_\Delta$ (see Remark~\ref{rem:ActionKerMDelta}). We obtain the one-parameter family of Ricci-flat metrics 
\begin{equation}
 \label{eqn:one_parameter_family}
g_1=1= g_2= g_3,\qquad g_4=y, \qquad g_5=-y,\qquad g_6= y^2,\qquad g_7=y.
\end{equation}
The Riemann tensor $R\colon\Lambda^2\g\to\End \g$ and its projection $R'\colon \Lambda^2\g'\to \End \g'$
satisfy
\[g(R,R) = \frac{1}{2}  y+y^{2}+1, \quad g(R',R') = - y^{2}- y+\frac{13}{8};\]
this proves that the metrics~\eqref{eqn:one_parameter_family} are pairwise nonisometric.
\end{example}
\begin{example}
Among $2$-step nice nilpotent Lie algebras of dimension $9$, the only ones that satisfy \ref{enum:condK}, \ref{enum:condH} and \ref{enum:condL} for $k=0$ are the Lie algebras in the one-parameter family 
\[\texttt{93:86} \quad (0,0,0,0,0,0,a e^{15}+e^{24}+e^{36} ,e^{13}+e^{25}+e^{46},e^{12}+e^{34}+e^{56}).\]
The generic element of $\ker \tran M_\Delta$ is
\[X=(-\frac{1}{2}  x,x,-\frac{1}{2}  x,x,-\frac{1}{2}  x,-\frac{1}{2}  x,-\frac{1}{2}  x,-\frac{1}{2}  x,x);\]
\ref{enum:condH} is trivially satisfied for $x\neq0$, and \ref{enum:condP} gives $64a^2=1$. Thus, there are two nice Lie algebras in this family that admit a diagonal Ricci-flat metric, with admissible signatures determined by \ref{enum:condL}, i.e.
\[\tfrac12\mathbf{S}=\{125,346\}.\]
In particular, we do not obtain Ricci-flat metrics of Lorentzian signature. 
In fact, it was proved in  \cite{GuediriBinAsfour} that Ricci-flat Lorentzian metrics on $2$-step nilpotent Lie algebras have degenerate center; diagonal metrics on a nice Lie algebra never have this property.

We note that these metrics are not $\ad$-invariant, i.e. they do not satisfy
\[\langle[x, y], z\rangle+\langle y,[x, z]\rangle= 0, \quad x, y, z\in\g.\]
In fact, a diagonal metric on a $2$-step nice Lie algebra cannot be $\ad$-invariant, as one can see by taking $x$ and $y$ to be elements of the nice basis with $z=[x,y]\neq0$. This is consistent with \cite[Corollary 2.6]{delBarco:LieAlgebrasAdmitting}.
\end{example}

\begin{example}
\label{example:NonAdInvariant}
Let $\g$ be the $10$-dimensional nice Lie algebra with structure equations given by
\[(0,0,0,0,0,0,0,0, e^{12}+e^{34}+e^{56}+e^{78},e^{15}+e^{26}+e^{37}+e^{48}).\]
This Lie algebra is also $2$-step nilpotent, and \ref{enum:condK} implies
\[X=(x,-x,x,-x,-x,-x,x,x);\] 
\ref{enum:condH} and \ref{enum:condP} are trivially satisfied for $x\neq0$, and we only need to compute the admissible signatures using~\ref{enum:condL}. We obtain
\begin{multline*}
\tfrac12\mathbf{S}=\{169,160,259,250,389,380,479,470,\\
1237,1248,1345,1578,2346,2678,3567,4568\}
\end{multline*}
(where a zero stands for the index $10$). To obtain an explicit metric, we need to solve the system $e^{M_{\Delta}}(g)=X$; normalizing to $x=1$, we obtain
\[\frac{g_9}{g_1 g_2}=-\frac{g_9}{g_3 g_4}=\frac{g_9}{g_5 g_6}=-\frac{g_9}{g_7 g_8}=-\frac{g_{0}}{g_1 g_5}=-\frac{g_{0}}{g_2 g_6}=\frac{g_{0}}{g_3g_7}=\frac{g_{0}}{g_4 g_8}=1\]
which has solution:
\begin{gather*}
 g_1=\pm g_6,\ g_2=\pm g_5,\ g_4=-\frac{g_5 g_6}{g_3},\\
\ g_7=\mp\frac{g_5 g_6}{g_3},\ g_8=\pm g_3,\ g_9=g_5g_6,\ g_{0}=\mp g_5g_6.
\end{gather*}
This defines a $3$-parameter family of (normalized) Ricci-flat metrics, and since $\ker M_{\Delta}$ has dimension 3, it gives an essentially unique Ricci-flat metric (see  Remark~\ref{rem:ActionKerMDelta}).
\end{example}

\begin{example}
\label{example:ParaComplexRicciFlat}
To illustrate the relation between $\sigma$-diagonal metrics and parahermitian geometry, consider the Lie algebra $\g$ of Example~\ref{example:NonAdInvariant}, which admits the order $2$, fixed-point-free automorphism $\sigma=(13) (27) (45) (68)(90)$ (as before, zero represents the index $10$).  The $\sigma$-invariant vectors in $\ker \tran M_\Delta$ have the form \[X=( x, - x,  x, - x, - x, -x,  x, x).\] Note that \ref{enum:condH} holds trivially for $x\neq0$. Since $\logsign c=\logsign \tilde{c}=[0]$, it is easy to find $\delta$ satisfying \ref{enum:condLsigma}, i.e. $\logsign(X)= M_{\Delta,2}\delta$; thus, there exists a $\sigma$-diagonal Ricci-flat metric with signature $(5,5)$. In fact, we obtain the following list of admissible signatures:
\[\mathbf{S}_\sigma(5,5)=\{1237,1345,2678,4568,123790,134590,267890,456890\}.\]
By Remark~\ref{rem:ParaComplexSigmaNotIntegrable}, or a direct computation, the paracomplex structure defined by $\sigma$ is not integrable.

However, we can choose several almost paracomplex structures adapted to the nice basis, either integrable or not. For example, consider the almost paracomplex structure $K$ such that the eigenspace associated to the eigenvalue $+1$ is $\g^+_{K}=\Span{e_1,e_4,e_6,e_7,e_9}$ whilst the eigenspace associated to $-1$ is $\g^-_{K}=\Span{e_2,e_3,e_5,e_8,e_0}$. We note that $K$ is integrable, and by Remark~\ref{rem:SigmaMetricParahermitian} any $\sigma$-diagonal metric $g$ is compatible with the paracomplex structure $K$, defining a parahermitian structure $(K,g)$.

We recall that a parahermitian structure can be viewed as a $\GL(5,\R)$-structure and the Ricci tensor can be computed using the intrinsic torsion of a reduction to $\SL(5,\R)$ obtained by fixing a volume of one eigendistribution (refer to \cite{ContiRossi:ricci} for more details). In particular, the $\SL(5,\R)$ intrinsic torsion can be split into ten invariant components:
\[\R^5\otimes\frac{\so(5,5)}{\Sl(5,\R)}=W_1\oplus W_2\oplus W_3\oplus W_4\oplus W_5\oplus W_6\oplus W_7\oplus W_8\oplus W^{1,0}\oplus W^{0,1}.\]

To compute explicitly the intrinsic torsion, let $g$ be the generic  $\sigma$-invariant metric that solves the equation $e^{M_{\Delta}}(g)=X$ for a generic $\sigma$-invariant $X$.
Then the nonzero components of the metric tensor are:
\[g_{27}=-g_{45},\quad g_{68}=-g_{13},\quad g_{90}=-x g_{13} g_{45}.\]
We fix an orthonormal basis by setting
\begin{equation}
 \label{eqn:ehat}
\begin{gathered}
\hat{e}_1 = \frac{e_1} {g_{13}},\quad \hat{e}_2 = \frac{e_2}{-g_{45}},\quad \hat{e}_3=e_{3},\quad \hat{e}_4 = \frac{e_4} {g_{45}},\quad\hat{e}_5 = e_5,\\
\hat{e}_6 = \frac{e_6}{-g_{13}},\quad \hat{e}_7=e_7,\quad \hat{e}_8=e_8,\quad \hat{e}_9 = \frac{e_9}{-x g_{13} g_{45}},\quad \hat{e}_{0}=e_{0},
\end{gathered}
\end{equation}
and consider the $\SL(5,\R)$-reduction defined by the volume form $\hat e^{14679}$. In this case, the $\SL(5,\R)$-structure has intrinsic torsion $\tau_3+\tau_7\in W_3\oplus W_7$, and all the other components vanish. 

We have: 
\begin{align*}
\tau_3&=-\frac{x}{2}\hat{e}^1\otimes\hat{e}^{20}-\frac{g_{13} x}{2} \hat{e}^{4}\otimes\hat{e}^{30}+\frac{g_{45} x}{2}\hat{e}^{6}\otimes\hat{e}^{50}+\frac{g_{13}g_{45} x}{2}\hat{e}^7\otimes\hat{e}^{80},\\
f_3&= x \hat{e}^{120} -g_{13} x \hat{e}^{340}+ g_{45} x \hat{e}^{560} -g_{13} g_{45} x \hat{e}^{780},\\
\tau_7&=-\frac{1}{2 g_{13}g_{45}}\hat{e}^{2}\otimes\hat{e}^{69}-\frac{1}{2}\hat{e}^{3}\otimes\hat{e}^{79}+\frac{1}{2 g_{13}}\hat{e}^5\otimes\hat{e}^{19}+\frac{1}{2g_{45}}\hat{e}^8\otimes\hat{e}^{49},\\
f_7&= -\frac{1}{g_{13}}\hat{e}^{159} -\frac{1}{g_{13} g_{45}} \hat{e}^{269} -\hat{e}^{379} -\frac{1}{g_{45}}\hat{e}^{489};
\end{align*}
where $f_3,f_7$ are differential forms uniquely associated to $\tau_3$ and $\tau_7$.  As the intrinsic torsion lies in $W_3\oplus W_7$, by \cite[Theorem 5.4]{ContiRossi:ricci} the Ricci tensor has the form
\begin{multline*}
\Ric=3[\Lambda(df_3+\partial(\tau_7)\hook f_3)]_{\Sl(5,\R)}-2F(\tau_3,\tau_7)\\
+ \epsilon \left(\Ric (-\partial(\tau_3)\hook \tau_3)\right)+\epsilon \left(\Ric (-\partial(\tau_7)\hook \tau_7)\right).
\end{multline*}
A computation shows that in this case the only nonzero components are
\begin{align*}
3[\Lambda(\partial(\tau_7)\hook f_3)]_{\Sl(5,\R)}
=&\ 2F(\tau_3,\tau_7)\\
=&\ \frac{x}{2} (\hat{e}^1\otimes \hat{e}^3-\hat{e}^3\otimes \hat{e}^1)+\frac{x}{2}(\hat{e}^2\otimes \hat{e}^7-\hat{e}^7\otimes \hat{e}^2) \\
&+\frac{x}{2}(-\hat{e}^4\otimes \hat{e}^5 +\hat{e}^5\otimes \hat{e}^4) +\frac{x}{2}(\hat{e}^6\otimes \hat{e}^8 -\hat{e}^8\otimes \hat{e}^6),
\end{align*}
proving that $\Ric=0$, as we already knew from Theorem~\ref{thm:sigmacompatible}. We observe that the Riemann curvature is not zero.

It is possible to define a different almost paracomplex structure which is not integrable. For example, let
$\tilde{K}$ be the almost paracomplex structure such that $\g_{\tilde{K}}^+=\Span{e_1,e_2,e_3,e_4,e_9}$ and $\g^-_{\tilde{K}}=\Span{e_5,e_6,e_7,e_8,e_0}$; note that $\tilde{K}$ is not integrable. As before, all $\sigma$-diagonal metrics $g$ are compatible with  $\tilde{K}$, giving an almost parahermitian structure. Using the orthonormal basis \eqref{eqn:ehat} and the volume form $\hat e^{12349}$, the nonzero components of the intrinsic torsion are
\begin{align*}
\tau_1&=-\frac{1}{6 g_{13} g_{45}}\hat{e}^{269},\\
\tau_2&=-\frac{1}{3 g_{13} g_{45}}\hat{e}^2\otimes\hat{e}^{69} + \frac{1}{3 g_{13} g_{45}}\hat{e}^6\otimes\hat{e}^{29} + \frac{2}{3 g_{13} g_{45}}\hat{e}^9\otimes\hat{e}^{26},\\
\tau_3&= -\frac{g_{13} x}{2}\hat{e}^4\otimes\hat{e}^{30} + \frac{g_{45} x}{2}\hat{e}^6\otimes\hat{e}^{50} + \frac{1}{2} \hat{e}^9\otimes \hat{e}^{37},\\
\tau_5&=\frac{g_{13}g_{45} x}{6}\hat{e}^{780},\\
\tau_6&=\frac{g_{13} g_{45} x}{3}\hat{e}^7\otimes\hat{e}^{80}-\frac{g_{13}g_{45} x}{3}\hat{e}^{8}\otimes\hat{e}^{70} + \frac{-2 g_{13} g_{45} x}{3} \hat{e}^{0}\otimes\hat{e}^{78},\\
\tau_7&=\frac{1}{2 g_{13}}\hat{e}^5\otimes\hat{e}^{19}+\frac{1}{2 g_{45}}\hat{e}^{8}\otimes\hat{e}^{49}+\frac{x}{2}\hat{e}^{0}\otimes\hat{e}^{12}.
\end{align*}
Using the formula in \cite{ContiRossi:ricci} we recover  that the Ricci is zero.
\end{example}

\bibliographystyle{plain}
\bibliography{ricciflat}

\def\cprime{$'$}
\begin{thebibliography}{10}

\bibitem{AncocheaCampoamor}
J.M. Ancochea and R.~Campoamor.
\newblock Characteristically nilpotent {L}ie algebras: a survey.
\newblock {\em Extracta Mathematica}, 16(2):153--210, 2001.

\bibitem{Arroyo}
R.~M. Arroyo.
\newblock Filiform nilsolitons of dimension 8.
\newblock {\em Rocky Mountain J. Math.}, 41(4):1025--1043, 2011.

\bibitem{delBarcoContiRossi:SigmaType}
D.~Conti, V.~del Barco, and F.~A. Rossi.
\newblock {D}iagram involutions and homogeneous {R}icci-flat metrics.
\newblock In preparation.

\bibitem{ContiRossi:EinsteinNice}
D.~Conti and F.~A. Rossi.
\newblock Indefinite {E}instein metrics on nice {L}ie groups.
\newblock arXiv:1805.08491.

\bibitem{ContiRossi:ricci}
D.~Conti and F.~A. Rossi.
\newblock The {R}icci tensor of almost parahermitian manifolds.
\newblock {\em Ann. Global Anal. Geom.}, 53(4):467--501, JUN 2018.

\bibitem{ContiRossi:Construction}
D.~Conti and F.~A. Rossi.
\newblock {C}onstruction of nice nilpotent {L}ie groups.
\newblock {\em Journal of Algebra}, 525:311 -- 340, 2019.

\bibitem{ContiRossi:EinsteinNilpotent}
D.~Conti and F.~A. Rossi.
\newblock Einstein nilpotent {L}ie groups.
\newblock {\em J. Pure Appl. Algebra}, 223(3):976--997, 2019.

\bibitem{delBarco:LieAlgebrasAdmitting}
V.~del Barco.
\newblock Lie algebras admitting symmetric, invariant and nondegenerate
  bilinear forms.
\newblock arXiv:1602.08286.

\bibitem{DelBarcoOvando}
V.~del Barco and G.~P. Ovando.
\newblock Isometric actions on pseudo-{R}iemannian nilmanifolds.
\newblock {\em Ann. Global Anal. Geom.}, 45(2):95--110, 2014.

\bibitem{Favre}
Gabriel Favre.
\newblock Une alg\`ebre de {L}ie caract\'{e}ristiquement nilpotente de
  dimension {$7$}.
\newblock {\em C. R. Acad. Sci. Paris S\'{e}r. A-B}, 274:A1338--A1339, 1972.

\bibitem{FernandezFreibertSanchez}
M.~Fernandez, M.~Freibert, and J.~Sanchez.
\newblock A 7-dimensional nilmanifold with a non {R}icci-flat {E}instein
  pseudo-metric.
\newblock In preparation, 2019.

\bibitem{FernandezCulma}
E.~A. Fern\'andez-Culma.
\newblock Classification of 7-dimensional {E}instein nilradicals.
\newblock {\em Transform. Groups}, 17(3):639--656, 2012.

\bibitem{FinoLujan:TorsionFreeG22}
A.~Fino and I.~Luj\'an.
\newblock Torsion-free {$G^*_{2(2)}$}-structures with full holonomy on
  nilmanifolds.
\newblock {\em Adv. Geom.}, 15(3):381--392, 2015.

\bibitem{globke}
W.~Globke and Y.~Nikolayevsky.
\newblock Compact pseudo-{R}iemannian homogeneous {E}instein manifolds of low
  dimension.
\newblock {\em Differential Geom. Appl.}, 54(part B):475--489, 2017.

\bibitem{Gong}
M.-P. Gong.
\newblock {\em Classification of nilpotent {L}ie algebras of dimension 7 (over
  algebraically closed fields and {R})}.
\newblock ProQuest LLC, Ann Arbor, MI, 1998.
\newblock Thesis (Ph.D.)--University of Waterloo (Canada).

\bibitem{GuediriBinAsfour}
M.~Guediri and M.~Bin-Asfour.
\newblock Ricci-flat left-invariant {L}orentzian metrics on 2-step nilpotent
  {L}ie groups.
\newblock {\em Arch. Math. (Brno)}, 50(3):171--192, 2014.

\bibitem{Heber:noncompact}
J.~Heber.
\newblock Noncompact homogeneous {E}instein spaces.
\newblock {\em Invent. Math.}, 133(2):279--352, 1998.

\bibitem{KadiogluPayne:Computational}
H.~Kadioglu and T.~L. Payne.
\newblock Computational methods for nilsoliton metric {L}ie algebras {I}.
\newblock {\em J. Symbolic Comput.}, 50:350--373, 2013.

\bibitem{Kath:NilpotentMetric}
I.~Kath.
\newblock Nilpotent metric {L}ie algebras of small dimension.
\newblock {\em J. Lie Theory}, 17(1):41--61, 2007.

\bibitem{Lauret}
J.~Lauret.
\newblock A canonical compatible metric for geometric structures on
  nilmanifolds.
\newblock {\em Ann. Global Anal. Geom.}, 30(2):107--138, 2006.

\bibitem{Lauret:niceEinstein}
J.~Lauret.
\newblock Einstein solvmanifolds and nilsolitons.
\newblock In {\em New developments in {L}ie theory and geometry}, volume 491 of
  {\em Contemp. Math.}, pages 1--35. Amer. Math. Soc., Providence, RI, 2009.

\bibitem{Lauret:Einstein_solvmanifolds}
J.~Lauret.
\newblock Einstein solvmanifolds are standard.
\newblock {\em Ann. of Math. (2)}, 172(3):1859--1877, 2010.

\bibitem{LauretWill:Einstein}
J.~Lauret and C.~Will.
\newblock Einstein solvmanifolds: existence and non-existence questions.
\newblock {\em Math. Ann.}, 350(1):199--225, 2011.

\bibitem{Malcev}
A.~I. Malcev.
\newblock On a class of homogeneous spaces.
\newblock {\em Amer. Math. Soc. Translation}, 1951(39):33, 1951.

\bibitem{Milnor:curvatures}
J.~Milnor.
\newblock Curvatures of left invariant metrics on {L}ie groups.
\newblock {\em Advances in Math.}, 21(3):293--329, 1976.

\bibitem{Nikolayevsky:EinsteinDerivation}
Y.~Nikolayevsky.
\newblock Einstein solvmanifolds with a simple {E}instein derivation.
\newblock {\em Geom. Dedicata}, 135:87--102, 2008.

\bibitem{Nikolayevsky:FreeNilradical}
Y.~Nikolayevsky.
\newblock Einstein solvmanifolds with free nilradical.
\newblock {\em Ann. Global Anal. Geom.}, 33(1):71--87, 2008.

\bibitem{Nikolayevsky}
Y.~Nikolayevsky.
\newblock Einstein solvmanifolds and the pre-{E}instein derivation.
\newblock {\em Trans. Amer. Math. Soc.}, 363(8):3935--3958, 2011.

\bibitem{Payne:Methods}
T.~L. Payne.
\newblock Methods for parametrizing varieties of {L}ie algebras.
\newblock {\em J. Algebra}, 445:1--34, 2016.

\bibitem{Petrov:EinsteinSpaces}
A.~Z. Petrov.
\newblock {\em Einstein spaces}.
\newblock Translated from the Russian by R. F. Kelleher. Translation edited by
  J. Woodrow. Pergamon Press, Oxford-Edinburgh-New York, 1969.

\bibitem{RossiPhD}
F.~A. Rossi.
\newblock {\em \textbf{D}-complex structures: cohomological properties and
  deformations}.
\newblock PhD thesis, Universit\`{a} degli Studi di Milano - Bicocca, 2013.
\newblock \texttt{http://hdl.handle.net/10281/41976}.

\bibitem{Will:RankOne}
C.~Will.
\newblock Rank-one {E}instein solvmanifolds of dimension 7.
\newblock {\em Differential Geom. Appl.}, 19(3):307--318, 2003.

\end{thebibliography}

\small\noindent Dipartimento di Matematica e Applicazioni, Universit\`a di Milano Bicocca, via Cozzi 55, 20125 Milano, Italy.\\
\texttt{diego.conti@unimib.it}\\
\texttt{federico.rossi@unimib.it}
\end{document}